\documentclass[12pt]{amsart}
\usepackage{amsmath}
\usepackage{amssymb}
\usepackage{color}
\usepackage{soul}
\usepackage{ifthen}

\newtheorem{propo}{Proposition}[section]

\newtheorem{theor}[propo]{Theorem}
\newtheorem{lemma}[propo]{Lemma}
\theoremstyle{definition}
\newtheorem{defin}[propo]{Definition}
\newtheorem{examp}[propo]{Example}

\theoremstyle{remark}
\newtheorem{remar}[propo]{Remark}

\newtheorem{conve}[propo]{Convention}

\numberwithin{equation}{section}

\renewcommand{\_}[1]{_{(#1)}}

\newcommand{\aaaU }{\Omega }
\newcommand{\actl }{\boldsymbol{\cdot}}
\newcommand{\ad }{\mathrm{ad}}
\newcommand{\al }{\alpha }
\newcommand{\antip }{S}
\newcommand{\Aut }{\mathrm{Aut}}
\newcommand{\barcX }{\overline{\cX }}

\newcommand{\bE }{\bar{E}}
\newcommand{\bF }{\bar{F}}
\newcommand{\bfun }[1]{b^{#1}}

\newcommand{\bnd }{b}

\newcommand{\cC }{\mathcal{C}}
\newcommand{\cF }{\mathcal{F}}
\newcommand{\cG }{\mathcal{G}}
\newcommand{\charUz}{\Hom (\cU ^0,\fieenz )}

\newcommand{\cI }{\mathcal{I}}

\newcommand{\coal }{\delta }
\newcommand{\copr }{\varDelta }
\newcommand{\coun }{\varepsilon }

\newcommand{\cU }{\mathcal{U}}
\newcommand{\cV }{\mathcal{V}}
\newcommand{\cX }{\mathcal{X}}

\newcommand{\derK }{\partial ^K}
\newcommand{\derL }{\partial ^L}
\newcommand{\End }{\mathrm{End}}
\newcommand{\fch }[1]{\mathrm{ch}\,#1}
\newcommand{\fie }{\Bbbk }
\newcommand{\Fie }{\bar{\fie }}
\newcommand{\fiee }{\mathbb{K}}
\newcommand{\fieenz }{\fiee ^\times }
\newcommand{\fienz }{\fie ^\times }
\newcommand{\Fienz }{\Fie ^\times }
\newcommand{\fin }{\mathrm{fin}}
\newcommand{\Frac}{\mathrm{Frac}}

\newcommand{\HCmap }[1]{\theta ^{#1}}
\newcommand{\Hom }{\mathrm{Hom}}
\newcommand{\id}{\operatorname{id}}

\newcommand{\im}{\operatorname{Im}}
\newcommand{\lact }{\boldsymbol{.}}
\newcommand{\lag }{\mathfrak{g}}

\newcommand{\LT }{T}

\newcommand{\maxspec }{\mathrm{maxspec}\,}

\newcommand{\NA }[1]{\mathfrak{B}(#1)}

\newcommand{\ndN }{\mathbb{N}}

\newcommand{\ndZ }{\mathbb{Z}}
\newcommand{\Ob }{\mathrm{Ob}}
\newcommand{\op }{^\mathrm{op}}

\newcommand{\ot }{\otimes }

\newcommand{\PF }{P}

\newcommand{\qfact }[2]{(#1)^!_{#2}}
\newcommand{\qnum }[2]{(#1)_{#2}}

\newcommand{\rfl }{r}
\newcommand{\rhomap }[1]{\rho ^{#1}}

\newcommand{\rsC}{\mathcal{R}}

\newcommand{\s }{\sigma }
\newcommand{\sdot }{\dot{\sigma }}
\newcommand{\Shf }{\mathrm{Sh}}
\newcommand{\sHp }{\eta }

\newcommand{\To }{\mathbb{T}}

\newcommand{\ula }{\underline{a}}

\newcommand{\ulm }{\underline{m}}
\newcommand{\Uz }{\cU ^0}

\newcommand{\VT }{t}
\newcommand{\VTM }{\hat{T}}
\newcommand{\Wg }{\mathcal{W}}
\newcommand{\YD }[1][ ]{Yetter--Drinfel'd#1}

\newcommand{\ydZI }{ {}_{\fie \ndZ ^I}^{\fie \ndZ ^I}\mathcal{YD}}
\newcommand{\ZIch }[1]{\zeta ^{#1}}
\newcommand{\ZIdual }{\widehat{\ndZ ^I}}

\title{
Drinfel'd doubles and Shapovalov determinants}

\author{I.~Heckenberger}
\thanks{I.\,H.~is supported by the German Research Foundation (DFG) via
a Heisenberg fellowship}
\address{Istv\'an Heckenberger,
Universit\"at zu K\"oln,
Mathematisches Institut,
Weyertal 86--90,
D-50931 K\"oln, Germany}
\email{i.heckenberger@googlemail.com}

\author{H.~Yamane}
\address{Hiroyuki Yamane,
Department of Pure and Applied Mathematics,
Graduate School of Information Science
and Technology, Osaka University, Toyonaka 560-0043,
Japan}
\email{yamane@ist.osaka-u.ac.jp}

\begin{document}

\maketitle

\begin{abstract}
  The Shapovalov determinant for a class of pointed Hopf
  algebras is calculated, including quantized enveloping algebras,
  Lusztig's small quantum groups, and quantized Lie superalgebras.
  Our main tools are root systems, Weyl groupoids, and Lusztig type
  isomorphisms. We elaborate powerful novel techniques
  for the algebras at
  roots of unity, and pass to the general case using
  a density argument.


\end{abstract}

\section{Introduction}

We study finite-dimensional representations of a large class
of Hopf algebras $U(\chi )$, where $\chi $ is a bicharacter on
$\ndZ ^I$ for some finite index set $I$. These algebras emerged from
a program of Andruskiewitsch and Schneider to classify pointed Hopf
algebras \cite{a-AndrSchn98}, \cite{p-Heck07b}.
Prominent examples are
quantized enveloping algebras of semisimple Lie algebras,
where the deformation parameter is not a root of $1$,
and Lusztig's (finite-dimensional) small quantum groups,
see Sect.~\ref{sec:Uqg}.
Other relevant examples are quantized enveloping algebras of Lie
superalgebras, see \cite{a-KhorTol91} and \cite{a-Yam99,a-Yam99e},
and Drinfeld doubles of bosonizations of
Nichols algebras of diagonal type classified in \cite{p-Heck06b}.

Our main combinatorial tools towards the study of representations
are the root system and the Weyl
groupoid associated to $\chi $. For quantized enveloping algebras
of semisimple Lie algebras
the Weyl groupoid is nothing but the Weyl group of the Lie algebra.
The main concern of this paper is the determination of the Shapovalov
determinants for all algebras $U(\chi )$ with finite root system.
We obtain a natural analog of
Shapovalov's original formula (for complex semisimple
Lie algebras) as a product of linear factors.
For our approach we need that
$\chi (\beta ,\beta )\not=1$ for all positive roots $\beta $.
This assumption is fulfilled for the special cases mentioned above.

The generality of our setting forces us to understand the
representation theory of algebras $U(\chi )$, where many values of
$\chi $ are roots of $1$. We turn this bondage into a promising
leading principle of our approach.
We concentrate first on those bicharacters, which take values in the
set of roots of $1$. In this case the positive and negative parts of
$U(\chi )$ are finite-dimensional algebras. For these algebras we
develop special techniques,
which are very different from the usual ones
for semisimple Lie algebras,
based on reflections in the Weyl groupoid,
With these techniques we are able to
characterize easily the irreducibility of Verma modules. The
characterization leads quickly to a formula for the Shapovalov
determinants of Verma modules by a variant of the usual density
argument. In a next step we extend our results to more general
bicharacters by a new density argument. This is possible because of
our good knowledge of the root system of bicharacters.

The history of Shapovalov determinants started with Shapovalov's work
\cite{a-Shapov72}, where he defined a bilinear form on Verma modules
and determinants on homogeneous subspaces to obtain information
on the reducibility of Verma modules. These structures have been
generalized by Kac and Kazhdan \cite{a-KK79} to symmetrizable
Kac-Moody algebras and by Kac \cite{inb-Kac79} \cite{inp-Kac86}
to Lie superalgebras with symmetrizable Cartan matrix.
Shapovalov determinants have been calculated for
quantized enveloping algebras by de Concini and Kac \cite{inp-dCK90}
and for quantized Kac-Moody algebras by Joseph \cite{b-Joseph}.
For Lusztig's small quantum groups
Kumar and Letzter \cite{a-KumLetz97}
factorize the Shapovalov determinants under the assumption that
the deformation parameter $q$
has prime order and the base field is a cyclotomic field.
Shapovalov determinants have been considered recently in
various contexts, see for example \cite{a-BrunKles02},
\cite{inp-GorSerg05}, \cite{a-Gorel06},
\cite{a-Alek05}, \cite{a-Hill08}.
Our approach yields in particular an entirely new proof of the
formula of de Concini and Kac.
In Sect.~\ref{sec:Uqg} we improve Kumar's and Letzter's result
by allowing arbitrary base fields and arbitrary orders of $q$.

The paper is organized as follows. In Sect.~\ref{sec:prelims} the
axioms of Cartan schemes, Weyl groupoids, and root systems are
recalled. The Weyl groupoid of a bicharacter fits into this framework.
Besides collecting the most important facts we introduce a character
$\rho ^\chi $ on $\ndZ ^I$ which will play a similar role as
the linear form $2\rho $ on the root lattice.
In Sect.~\ref{sec:DD} the definition and properties of the Drinfeld
doubles $U(\chi )$ are recalled. Sect.~\ref{sec:Lusztig} deals with
Lusztig type isomorphisms between two (usually different)
Drinfeld doubles. With Thm.~\ref{th:PBWtau}
we establish a Lusztig type PBW basis of these algebras.
Moreover, in Thm.~\ref{th:EErel} we develop important properties for
$q$-commutators of root vectors. In Sect.~\ref{sec:Verma} we start to
study Verma modules and special maps between them.
Prop.~\ref{pr:VTMiso}
gives a criterion for bijectivity of such maps, and Prop.~\ref{pr:M=L}
identifies irreducible Verma modules.
In Sect.~\ref{sec:shapdet} we study Shapovalov determinants following
the approach in \cite{b-Joseph}. Here our main result is
Thm~\ref{th:Shapdet}, which gives a formula for the Shapovalov
determinant of $U(\chi )$, where all values of $\chi $ are roots of
$1$, and the root system of $\chi $ is finite. Then we pass to more
general bicharacters: Thm.~\ref{th:Shapdet2} states a similar result
for bicharacters with finite root system. We conclude the paper with
the adaptation of our formulas to quantized enveloping algebras and
Lusztig's small quantum groups
in Sect.~\ref{sec:Uqg}, and with some commutative
algebra in the Appendix.

\section{Preliminaries}
\label{sec:prelims}

Let $\fie $ be a field and $\fienz =\fie \setminus \{0\}$.
For all $n\in \ndN _0$ and $q\in \fienz $
let
$$\qnum nq=\sum _{j=0}^{n-1}q^j,\qquad \qfact nq=\prod _{j=1}^n\qnum jq,$$
where $\qfact 0q=1$.
For any finite set $I$ let
$\{\al _i\,|\,i\in I\}$ be the standard basis of the free
$\ndZ $-module $\ndZ ^I$.

\subsection{Cartan schemes, Weyl groupoids, and root systems}
\label{ssec:CS}

The combinatorics of a Drinfel'd double of a Nichols algebra of diagonal type
is controlled to a large extent by its Weyl groupoid.
We use the language developed in
\cite{p-CH08}. Substantial part of the theory was obtained first in
\cite{a-HeckYam08}. We recall the most important definitions and facts.

Let $I$ be a non-empty finite set.
By \cite[\S 1.1]{b-Kac90} a generalized Cartan matrix
$C=(c_{ij})_{i,j\in I}$
is a matrix in $\ndZ ^{I\times I}$ such that
\begin{enumerate}
  \item[(M1)] $c_{ii}=2$ and $c_{jk}\le 0$ for all $i,j,k\in I$ with
    $j\not=k$,
  \item[(M2)] if $i,j\in I$ and $c_{ij}=0$, then $c_{ji}=0$.
\end{enumerate}

\begin{defin} \label{de:CS}
  Let $I$ be a non-empty finite set,
  $A$ a non-empty set, $\rfl _i : A \to A$ a map for all $i\in I$,
  and $C^a=(c^a_{jk})_{j,k \in I}$ a generalized Cartan matrix
  in $\ndZ ^{I \times I}$ for all $a\in A$. The quadruple
  \[ \cC = \cC (I,A,(\rfl _i)_{i \in I}, (C^a)_{a \in A})\]
  is called a \textit{Cartan scheme} if
  \begin{enumerate}
  \item[(C1)] $\rfl _i^2 = \id$ for all $i \in I$,
  \item[(C2)] $c^a_{ij} = c^{\rfl _i(a)}_{ij}$ for all $a\in A$ and
    $i,j\in I$.
  \end{enumerate}
\end{defin}

\begin{examp}
  Let $A=\{a\}$ be a set with a single element, and let $C$ be a
  generalized Cartan matrix. Then $r_i=\id $ for all $i\in I$, and
  $\cC $ becomes a Cartan scheme.
\end{examp}

  One says that a Cartan scheme $\cC $ is \textit{connected}, if
  the group $\langle \rfl _i\,|\,i\in I\rangle \subset \Aut (A)$ acts
  transitively on $A$, that is, if
  for all $a,b\in A$ with $a\not=b$ there exist $n\in \ndN $
  and $i_1,i_2,\ldots ,i_n\in I$ such that $b=r_{i_n}\cdots r_{i_2}
  r_{i_1}(a)$.
  Two Cartan schemes $\cC =\cC (I,A,(\rfl _i)_{i\in I},(C^a)_{a\in A})$
  and $\cC '=\cC '(I',A',$
  $(\rfl '_i)_{i\in I'},({C'}^a)_{a\in A'})$
  are called
  \textit{equivalent}, if there are bijections $\varphi _0:I\to I'$
  and $\varphi _1:A\to A'$ such that
  \begin{align}\label{eq:equivCS}
    \varphi _1(\rfl _i(a))=\rfl '_{\varphi _0(i)}(\varphi _1(a)),
    \qquad
    c^{\varphi _1(a)}_{\varphi _0(i) \varphi _0(j)}=c^a_{i j}
  \end{align}
  for all $i,j\in I$ and $a\in A$.

  Let $\cC = \cC (I,A,(\rfl _i)_{i \in I}, (C^a)_{a \in A})$ be a
  Cartan scheme. For all $i \in I$ and $a \in A$ define $\s _i^a \in
  \Aut(\ndZ ^I)$ by
  \begin{align}
    \s _i^a (\al _j) = \al _j - c_{ij}^a\al _i \qquad
    \text{for all $j \in I$.}
    \label{eq:sia}
  \end{align}
  This map is a reflection. 
  The \textit{Weyl groupoid of} $\cC $
  is the category $\Wg (\cC )$ such that $\Ob (\Wg (\cC ))=A$ and
  the morphisms are generated by the maps
  $\s _i^a\in \Hom (a,\rfl _i(a))$ with $i\in I$, $a\in A$.
  Formally, for $a,b\in A$ the set $\Hom (a,b)$ consists of the triples
  $(b,f,a)$, where
  \[ f=\s _{i_n}^{\rfl _{i_{n-1}}\cdots \rfl _{i_1}(a)}\cdots
    \s _{i_2}^{\rfl _{i_1}(a)}\s _{i_1}^a \]
  and $b=\rfl _{i_n}\cdots \rfl _{i_2}\rfl _{i_1}(a)$ for some
  $n\in \ndN _0$ and $i_1,\ldots ,i_n\in I$.
  The composition is induced by the group structure of $\Aut (\ndZ ^I)$:
  \[ (a_3,f_2,a_2)\circ (a_2,f_1,a_1) = (a_3,f_2f_1, a_1)\]
  for all $(a_3,f_2,a_2),(a_2,f_1,a_1)\in \Hom (\Wg (\cC ))$.
  By abuse of notation we will write
  $f\in \Hom (a,b)$ instead of $(b,f,a)\in \Hom (a,b)$.
 
  The cardinality of $I$ is termed the \textit{rank of} $\Wg (\cC )$.
  A Cartan scheme is called \textit{connected} if its Weyl groupoid
  is connected.

Recall that a groupoid is a category such that all morphisms are
isomorphisms.
The Weyl groupoid $\Wg (\cC )$ of a Cartan scheme $\cC $ is a groupoid,
see \cite{p-CH08}. For all $i\in I$ and $a\in A$
the inverse of $\s _i^a$ is $\s _i^{r_i(a)}$.
If $\cC $ and $\cC '$ are equivalent Cartan schemes, then $\Wg (\cC )$ and
$\Wg (\cC ')$ are isomorphic groupoids.

A groupoid $G$ is called \textit{connected},
if for each $a,b\in \Ob (G)$ the class $\Hom (a,b)$ is non-empty.
Hence $\Wg (\cC )$ is a connected groupoid if and only if $\cC $ is a
connected Cartan scheme.

\begin{defin} \label{de:RSC}
  Let $\cC =\cC (I,A,(\rfl _i)_{i\in I},(C^a)_{a\in A})$ be a Cartan
  scheme. For all $a\in A$ let $R^a\subset \ndZ ^I$, and define
  $m_{i,j}^a= |R^a \cap (\ndN_0\al _i + \ndN_0\al _j)|$ for all $i,j\in
  I$ and $a\in A$. We say that
  \[ \rsC = \rsC (\cC , (R^a)_{a\in A}) \]
  is a \textit{root system of type} $\cC $, if it satisfies the following
  axioms.
  \begin{enumerate}
    \item[(R1)]
      $R^a=R^a_+\cup - R^a_+$, where $R^a_+=R^a\cap \ndN_0^I$, for all
      $a\in A$.
    \item[(R2)]
      $R^a\cap \ndZ \al _i=\{\al _i,-\al _i\}$ for all $i\in I$, $a\in A$.
    \item[(R3)]
      $\s _i^a(R^a) = R^{\rfl _i(a)}$ for all $i\in I$, $a\in A$.
    \item[(R4)]
      If $i,j\in I$ and $a\in A$ such that $i\not=j$ and $m_{i,j}^a$ is
      finite, then
      $(\rfl _i\rfl _j)^{m_{i,j}^a}(a)=a$.
  \end{enumerate}
  If $\rsC $ is a root system of type $\cC $, then
  $\Wg (\rsC )=\Wg (\cC )$ is the \textit{Weyl groupoid of} $\rsC $.
  Further, $\rsC $ is called \textit{connected}, if $\cC $ is a connected
  Cartan scheme.
  If $\rsC =\rsC (\cC ,(R^a)_{a\in A})$ is a root system of type $\cC $
  and $\rsC '=\rsC '(\cC ',({R'}^a_{a\in A'}))$ is a root system of
  type $\cC '$, then we say that $\rsC $ and $\rsC '$ are \textit{equivalent},
  if $\cC $ and $\cC '$ are equivalent Cartan schemes given by maps $\varphi
  _0:I\to I'$, $\varphi _1:A\to A'$ as in Def.~\ref{de:CS}, and if
  the map $\varphi _0^*:\ndZ^I\to \ndZ^{I'}$ given by
  $\varphi _0^*(\al _i)=\al _{\varphi _0(i)}$ satisfies
  $\varphi _0^*(R^a)={R'}^{\varphi _1(a)}$ for all $a\in A$.
\end{defin}

There exist many interesting examples of root systems of type $\cC $ related
to semisimple Lie algebras, Lie superalgebras and Nichols algebras of diagonal
type, respectively. For further details and results we refer to
\cite{a-HeckYam08} and \cite{p-CH08}.

\begin{conve}\label{con:uind}
  In connection with Cartan schemes $\cC $,
  upper indices usually refer to elements of $A$.
  Often, these indices will be omitted if they are uniquely determined
  by the context. In particular,
  for any $w,w'\in \Hom (\Wg (\cC ))$ and $a\in A$,
  the notation
  $1_aw$ and $w'1_a$ means that $w\in \Hom (\underline{\,\,},a)$ and
  $w'\in \Hom (a,\underline{\,\,})$, respectively.
\end{conve}

A fundamental result about Weyl groupoids is the following theorem.

\begin{theor}\cite[Thm.\,1]{a-HeckYam08}\label{th:Coxgr}
  Let $\cC =\cC (I,A,(\rfl _i)_{i\in I},(C^a)_{a\in A})$
  be a Cartan scheme and $\rsC =\rsC (\cC ,(R^a)_{a\in A})$ a root system
  of type $\cC $.
  Let $\Wg $ be the abstract
  groupoid with $\Ob (\Wg )=A$ such that $\Hom (\Wg )$ is
  generated by abstract morphisms $s_i^a\in \Hom (a,\rfl _i(a))$,
  where $i\in I$ and $a\in A$, satisfying the relations
  \begin{align*}
    s_i s_i 1_a=1_a,\quad (s_j s_k)^{m_{j,k}^a}1_a=1_a,
    \qquad a\in A,\,i,j,k\in I,\, j\not=k,
  \end{align*}
  see Conv.~\ref{con:uind}.
  Here $1_a$ is the identity of the object $a$,
  and $(s_j s_k)^\infty 1_a$ is understood to be
  $1_a$. The functor $\Wg \to \Wg (\rsC )$, which is
  the identity on the objects, and on the set of
  morphisms is given by
  $s _i^a\mapsto \s_i^a$ for all $i\in I$, $a\in A$,
  is an isomorphism of groupoids.
\end{theor}

If $\cC $ is a Cartan scheme,
then the Weyl groupoid $\Wg (\cC )$ admits a length function $\ell :\Wg (\cC
)\to \ndN _0$ such that
\begin{align}
  \ell (w)=\min \{k\in \ndN _0\,|\,\exists i_1,\dots ,i_k\in I,a\in A:
  w=\s _{i_1}\cdots \s _{i_k}1_a\}
  \label{eq:ell}
\end{align}
for all $w\in \Wg (\cC )$.
If there exists a root system of type $\cC $,
then $\ell $ has very similar properties to the well-known length function
for Weyl groups, see \cite{a-HeckYam08}.

\begin{lemma}
  Let $\cC $ be a Cartan scheme and $\rsC $ a root system of type $\cC $.
  Let $a\in A$. Then $-c^a_{ij}=\max\{m\in \ndN _0\,|\,\al _j+m\al _i\in
  R^a_+\}$ for all $i,j\in I$ with $i\not=j$.
  \label{le:cm}
\end{lemma}

\begin{proof}
  By (C2) and (R3), $\s _i^{r_i(a)}(\al _j)=\al _j-c^a_{ij}\al _i\in R^a_+$.
  Hence $-c^a_{ij}\le \max\{m\in \ndN _0\,|\,\al _j+m\al _i\in R^a_+\}$.
  On the other hand, if $\al _j+m\al _i\in R^a_+$, then $\s _i^a(\al _j+m\al
  _i)=\al _j+(-c^a_{ij}-m)\al _i\in R^{r_i(a)}_+$ by (R3) and (R1), and hence
  $m\le -c^a_{ij}$. This proves the lemma.
\end{proof}

Let $\cC $ be a Cartan scheme and $\rsC $ a root system of type
$\cC $. We say that $\rsC $ is \textit{finite}, if $R^a$ is finite for
all $a\in A$.
The following lemmata are well-known for traditional root systems.

\begin{lemma} \cite[Lemma 2.11]{p-CH08}
  Let $\cC $ be a connected Cartan scheme and $\rsC $
  a root system of type $\cC $. The following are equivalent.
  \begin{enumerate}
    \item $\rsC $ is finite.
    \item $R^a$ is finite for at least one $a\in A$.
    \item $\Wg (\rsC )$ is finite.
  \end{enumerate}
  \label{le:Rfincond}
\end{lemma}

\begin{lemma} \cite[Cor.\,5]{a-HeckYam08}
  Let $\cC $ be a connected Cartan scheme and $\rsC $ a finite root system
  of type $\cC $. Then for all $a\in A$ there exist unique elements $b\in A$
  and $w\in \Hom (b,a)$ such that $|R^a_+|=\ell (w)\ge \ell (w')$
  for all $w'\in \Hom (b',a')$, $a',b'\in A$.
  \label{le:longestw}
\end{lemma}

\subsection{The Weyl groupoid of a bicharacter}
\label{ssec:Weylgroupoid}

Let $I$ be a non-empty finite set.
Recall that a bicharacter on $\ndZ ^I$ with values in $\fienz $ is a map
$\chi :\ndZ ^I\times \ndZ ^I\to \fienz $ such that
\begin{align}
	\chi (a+b,c)=&\chi (a,c)\chi (b,c),&
	\chi (c,a+b)=&\chi (c,a)\chi (c,b)
	\label{eq:bichar}
\end{align}
for all $a,b,c\in \ndZ ^I$.
Then $\chi (0,a)=\chi (a,0)=1$ for all $a\in \ndZ ^I$.
Let $\cX $ be the set of bicharacters on $\ndZ ^I$.  
If $\chi \in \cX $, then
\begin{align}
  \label{eq:chiop}
  \chi \op : & \ndZ ^I\times \ndZ ^I\to \fienz ,&
  \chi \op (a,b)=&\,\chi (b,a),\\
  \label{eq:chiinv}
  \chi ^{-1} : & \ndZ ^I\times \ndZ ^I\to \fienz ,&
  \chi ^{-1}(a,b)=&\,\chi (a,b)^{-1},
  \intertext{and for all $w\in \Aut _\ndZ (\ndZ ^I)$ the map}
  \label{eq:w*chi}
  w^*\chi : & \ndZ ^I\times \ndZ ^I\to \fienz ,&
  w^*\chi (a,b)=&\,\chi (w^{-1}(a),w^{-1}(b)),
\end{align}
are bicharacters on $\ndZ ^I$. The equation
\begin{align}
	\label{eq:w*functor}
	(ww')^*\chi =w^*(w'{}^*\chi )
\end{align}
holds for all $w,w'\in \Aut _{\ndZ }(\ndZ ^I)$ and all $\chi \in \cX $.

\begin{defin}\label{de:Cartan}
	Let $\chi \in \cX $, $p\in I$, and
	$q_{ij}=\chi (\al _i,\al _j)$ for all $i,j\in I$.
  We say that $\chi $ is $p$-\textit{finite}, if for all $j\in I$ there exists
  $m\in \ndN _0$ such that $\qnum{m+1}{q_{pp}}=0$ or
  $q_{pp}^m q_{pj}q_{jp}=1$.

  Assume that $\chi $ is $p$-finite.
  Let $c_{p p}^\chi =2$, and for all $j\in I\setminus \{p\}$ let
	$$ c_{pj}^\chi =-\min \{m\in \ndN _0 \,|\,
	(m+1)_{q_{pp}}(q_{pp}^m q_{pj} q_{jp}-1)=0\}. $$
  If $\chi $ is $i$-finite for all $i\in I$, then the matrix
  $C^\chi =(c_{ij}^\chi )_{i,j\in I}$ is called the
	\textit{Cartan matrix} associated to $\chi $.
  It is a generalized Cartan matrix, see Sect.~\ref{ssec:CS}.
\end{defin}

For all $p\in I$ and $\chi \in \cX $, where $\chi $ is $p$-finite, let
$\s _p^\chi \in \Aut _\ndZ (\ndZ ^I)$,
\begin{align*}
  \s _p^\chi (\al _j)=\al _j-c_{pj}^\chi \al _p
  \quad \text{for all $j\in I$.}
\end{align*}
Towards the definition of the Weyl groupoid of a
bicharacter, we define bijections $r_p:\cX \to \cX $
for all $p\in I$. Namely, let
\begin{align*}
  r_p: \cX \to \cX,\quad
  r_p(\chi )=
  \begin{cases}
    (\s _p^\chi )^*\chi & \text{if $\chi $ is $p$-finite,}\\
    \chi & \text{otherwise.}
  \end{cases}
\end{align*}
Let $p\in I$, $\chi \in \cX $,
$q_{ij}=\chi (\al _i,\al _j)$ for all $i,j\in I$.
If $\chi $ is $p$-finite, then
\begin{equation}
  \begin{aligned}
    r_p(\chi )(\al _p,\al _p)=&q_{p p}, &
    r_p(\chi )(\al _p,\al _j)=&q_{p j}^{-1}q_{p p}^{c_{pj}^\chi },\\
    r_p(\chi )(\al _i,\al _p)=&q_{i p}^{-1}q_{p p}^{c_{pi}^\chi },&
    r_p(\chi )(\al _i,\al _j)=&q_{i j} q_{i p}^{-c_{p j}^\chi }
    q_{p j}^{-c_{p i}^\chi } q_{p p}^{c_{pi}^\chi c_{p j}^\chi }
  \end{aligned}
  \label{eq:rpchi}
\end{equation}
for all $i,j\in I\setminus \{p\}$. 
It is a small exercise to check that then
$(\s _p^\chi )^*\chi $ is $p$-finite, and
\begin{align}\label{eq:rp2}
  c_{pj}^{r_p(\chi )}=c_{pj}^\chi \quad \text{for all $j\in I$},
  \qquad r_p^2(\chi )=\chi .
\end{align}
The reflections $r_p$, $p\in I$, generate a subgroup
\begin{align*}
  \cG =\langle r_p\,|\, p\in I\rangle
\end{align*}
of the group of bijections of the set $\cX $. For all $\chi \in \cX $ let
$\cG (\chi )$ denote the $\cG $-orbit of $\chi $ under the action of $\cG $. 

Let $\chi \in \cX $ such that $\chi '$ is $p$-finite for all
$\chi '\in \cG (\chi )$ and $p\in I$.
By Eq.~\eqref{eq:rp2} we obtain that
$$\cC (\chi )=
\cC (I,\cG (\chi ),(r_p)_{p\in I},
(C^{\chi '})_{\chi '\in \cG (\chi )})$$
is a connected Cartan scheme.
The Weyl groupoid of $\chi $ is then the Weyl groupoid of the Cartan scheme
$\cC (\chi )$ and is denoted by $\Wg (\chi )$. Clearly,
$\cC (\chi )=\cC (\chi ')$ and $\Wg (\chi )=\Wg (\chi ')$
for all $\chi '\in \cG (\chi )$.

\begin{examp}\label{ex:Cartan}
  Let $C=(c_{i j})_{i,j\in I}$ be a generalized Cartan matrix. Let
  $\chi \in \cX $, $q_{ij}=\chi (\al _i,\al _j)$ for all $i,j\in I$, and
  assume that $q_{ii}^{c_{ij}}=q_{ij}q_{ji}$ for all $i,j\in I$, and that
  $\qnum{m+1}{q_{ii}}\not=0$ for all $i\in I$ and $m\in \ndN _0$ with
  $m<\max \{-c_{ij}\,|\,j\in I\setminus \{i\}\}$.
  (The latter is not an essential assumption, since if it fails, then one can
  replace $C$ by another generalized Cartan matrix $\tilde{C}$,
  such that $\chi $ has this property with respect to $\tilde{C}$.)
  One says that $\chi $ is of \textit{Cartan type}.
  Then $\chi $ is $i$-finite for all $i\in I$,
  and $c_{ij}^\chi =c_{ij}$ for all $i,j\in I$. Eq.~\eqref{eq:rpchi} gives that
  \begin{align*}
    r_p(\chi )(\al _i,\al _i)=&q_{ii}=\chi (\al _i,\al _i),\\
    r_p(\chi )(\al _i,\al _j)\,r_p(\chi )(\al _j,\al _i)=&q_{ij}q_{ji}=
    r_p(\chi )(\al _i,\al _i)^{c_{pi}}
  \end{align*}
  for all $p,i,j\in I$.
  Hence $r_p(\chi )$ is again of Cartan type with the same Cartan matrix $C$.
  Thus $\chi '$ is $i$-finite for all $\chi '\in \cG (\chi )$ and
  $i\in I$.

  Let $C=(c_{ij})_{i,j\in I}$ be a symmetrizable generalized Cartan matrix,
  and for all $i\in I$ let $d_i\in \ndN $ such that
  $d_ic_{ij}=d_jc_{ji}$ for all $i,j\in I$.
  Let $q\in \fienz $ such that $\qnum{m+1}{q^{2d_i}}\not=0$
  for all $m\in \ndN _0$
  with $m<-c_{ij}$ for some $j\in I$.
  Define $\chi \in \cX $ by
  $\chi (\al _i,\al _j)=q^{d_ic_{ij}}$.
  Then $\chi $ is of Cartan type, hence $\chi $ is $p$-finite for all $p\in I$.
  Eq.~\eqref{eq:rpchi} implies that $r_p(\chi )=\chi $ for all $p\in I$,
  and hence $\cG (\chi )$ consists of precisely one element.
  In this case the Weyl groupoid $\Wg (\chi )$ is a group, which is precisely
  the Weyl group associated to the generalized Cartan matrix $C$.
  We will study this example in Sect.~\ref{sec:Uqg} under the
  assumption that $C$ is of finite type.
\end{examp}

\subsection{Roots}
\label{ssec:roots}

Let $\chi \in \cX $.
There exists a canonical root system of type $\cC (\chi)$ which we describe in
this subsection.
It is based on the construction of a restricted PBW basis
of Nichols algebras of diagonal type. Nichols algebras are braided
Hopf algebras defined by a universal property.
More details can be found in
\cite{inp-AndrSchn02} on braided Hopf algebras and Nichols algebras,
in \cite{a-Khar99} on the PBW
basis, and in \cite{a-Heck06a} on the root system.

Let $V\in \ydZI $ be a $|I|$-dimensional \YD module
of diagonal type. Let $\coal :V\to \fie \ndZ ^I\ot V$ and
$\actl :\fie \ndZ ^I\ot V\to V$ denote the left coaction and the left
action of $\fie \ndZ ^I$ on $V$, respectively. Fix a basis
$\{x_i\,|\,i\in I\}$ of $V$, elements $g_i$, where $i\in I$, and a
matrix $(q_{ij})_{i,j\in I}\in (\fienz)^{I\times I}$, such that
\[ \coal (x_i)=g_i\ot x_i,\quad g_i\actl x_j=q_{ij}x_j \quad
\text{for all $i,j\in I$.} \]
Assume that $\chi (\al _i,\al _j)=q_{ij}$ for all $i,j\in I$.
For all $\al \in \ndZ ^I$ define the ``bound function''
\begin{align}
  \bfun \chi (\al )=&
  \begin{cases}
    \min \{ m\in \ndN \,|\,
    \qnum{m}{\chi (\al ,\al )}=0\} &
    \text{if $\qnum{m}{\chi (\al ,\al )}=0$}\\
    & \text{for some $m\in \ndN $,}\\
    \infty & \text{otherwise}.
  \end{cases}
  \label{eq:height}
\end{align}
If $p\in I$ such that $\chi $ is $p$-finite, then
\begin{align}
  \bfun{r_p(\chi )}(\s _p^\chi (\al ))=\bfun \chi (\al )\quad
  \text{for all $\al \in \ndZ ^I$}
  \label{eq:hghtrpchi}
\end{align}
by Eq.~\eqref{eq:w*chi}.

The tensor algebra $T(V)$ admits a universal braided Hopf algebra quotient
$\NA V$, called the \textit{Nichols algebra of} $V$.
As an algebra, $\NA V$ has a unique $\ndZ ^I$-grading
\begin{align}
  \NA V=\oplus _{\al \in \ndZ ^I}\NA V _\al 
  \label{eq:NAVgrading}
\end{align}
such that $\deg x_i=\al _i$ for all $i\in I$.
This is also a coalgebra grading.
There exists a totally ordered index set $(L,\le )$ and a family
$(y_l)_{l\in L}$ of $\ndZ ^I$-homogeneous elements $y_l\in \NA V$ such that
the set
\begin{equation}
\begin{aligned}
  \{ y_{l_1}^{m_1}y_{l_2}^{m_2}\cdots y_{l_k}^{m_k}\,|\,
  &k\ge 0,\,l_1,\dots ,l_k\in L,\,l_1>l_2>\cdots >l_k,\\
  &m_i\in \ndN ,\,m_i<\bfun \chi (\deg y_{l_i})
  \quad \text{for all $i\in I$}\}
\end{aligned}
  \label{eq:PBWbasis}
\end{equation}
forms a vector space basis of $\NA V$.
The set
\begin{align}\label{eq:roots}
  R^\chi _+=\{\deg y_l\,|\,l\in L\}\subset \ndZ ^I
\end{align}
depends on the matrix
$(q_{ij})_{i,j\in I}$, but not on the choice of
the basis $\{x_i\,|\,i\in I\}$, the set $L$,
and the elements $g_i$, $i\in I$, and $y_l$, $l\in L$. Let
\[ R^\chi =R^\chi _+\cup -R^\chi _+. \]

\begin{theor}\cite[Thm.\,3.13]{p-Heck07b}
  \label{th:rschi}
  Let $\chi \in \cX $ such that $\chi '$ is $p$-finite for all $p\in
  I$.
  $\chi '\in \cG (\chi )$. Then
  $\rsC (\chi )=\rsC (\cC (\chi ), (R^{\chi '})_{\chi
  '\in \cG (\chi )})$ is a root system of type $\cC (\chi )$.
\end{theor}

Roots with finite bounds often play a distinguished role.
For all $\chi \in \cX $ let
\begin{align}
  R^\chi _{+\fin }=\{\beta \in R^\chi _+\,|\,
  \bfun \chi (\beta )<\infty \},\quad
  R^\chi _{+\infty }=R^\chi _+\setminus R^\chi _{+\fin }.
  \label{eq:R+fin}
\end{align}

We will use several finiteness properties of bicharacters.
\begin{align}
  \label{eq:X1}
  \cX _1=\{\chi \in \cX \,|\,&
  \text{$\chi $ is $p$-finite for all $p\in I$}\},\\
  \cX _2=\{\chi \in \cX \,|\,&
  \text{$\chi '$ is $p$-finite for all $\chi '\in \cG (\chi )$, $p\in I$}\},\\
  \cX _3=\{\chi \in \cX \,|\,&
  \text{$R^\chi $ is finite}\},\\
  \cX _4=\{\chi \in \cX \,|\,&
  \text{$R^\chi $ is finite, $R^\chi _+=R^\chi _{+\fin }$}\},\\
  \label{eq:X5}
  \cX _5=\{\chi \in \cX _4\,|\,&
  \text{$\chi (\al ,\al )\not=1$ for all $\al \in R^\chi _+$}\}.
\end{align}
Clearly, $\cX _i\supset \cX _j$ for $1\le i<j\le 5$.
By Eq.~\eqref{eq:height}, $\chi \in \cX _5$ if and only if
$R^\chi $ is finite and $\chi (\al ,\al )$ is a root of $1$ different
from $1$ for all $\al \in R^\chi _+$.

\begin{lemma} \label{le:equalrs}
  Let $\chi ,\chi '\in \cX _2$.

  (i) If $R^\chi _+=R^{\chi '}_+$, then
  $C^{w^*\chi }=C^{w^*\chi '}$ for all $w\in \Hom (\chi ,\underline{\,\,})
  \subset \Hom (\Wg (\chi ))$.

  (ii) Assume that $\chi ,\chi '\in \cX _3$.
  If
  $C^{w^*\chi }=C^{w^*\chi '}$ for all $w\in \Hom (\chi ,\underline{\,\,})
  \subset \Hom (\Wg (\chi ))$, then
  $R^\chi _+=R^{\chi '}_+$.
\end{lemma}

\begin{proof}
  By Thm.~\ref{th:rschi}, $\rsC (\chi )$ is a root system of type $\cC
  (\chi )$.

  (i) Assume that $R^\chi _+=R^{\chi '}_+$.
  Then $C^{\chi }=C^{\chi '}$ by Lemma~\ref{le:cm}.
  Therefore $\s _i^\chi =\s _i^{\chi '}$ in $\Aut (\ndZ ^I)$.
  Since $\chi ,\chi '\in \cX _2$, by induction
  it follows that $\s_{i_1}\cdots \s _{i_k}^\chi =\s _{i_1}\cdots
  \s _{i_k}^{\chi '}$ in $\Aut (\ndZ ^I)$ and
  $C^{(\s_{i_1}\cdots \s _{i_k}^\chi )^*\chi }=
  C^{(\s_{i_1}\cdots \s _{i_k}^{\chi '})^*\chi '}$
  for all $k\in \ndN _0$ and $i_1,\dots ,i_k\in I$.
  Hence
  $C^{w^*\chi }=C^{w^*\chi '}$ for all $w\in \Hom (\chi ,\underline{\,\,})
  \subset \Hom (\Wg (\chi ))$.

  (ii) Since $\chi \in \cX _3$,
  $R^\chi =\{w^{-1}(\al _i)\,|\,w\in \Hom
  (\chi ,\underline{\,\,})\subset \Hom (\Wg (\chi ))\}$
  by \cite[Prop.\,2.12]{p-CH08}. By assumption on the Cartan matrices,
  $\s_{i_1}\cdots \s _{i_k}^\chi =\s _{i_1}\cdots
  \s _{i_k}^{\chi '}$ in $\Aut (\ndZ ^I)$
  for all $k\in \ndN _0$ and $i_1,\dots,i_k\in I$.
  Hence $R^\chi =R^{\chi '}$, and the
  lemma holds by (R1).
\end{proof}

For our study of Drinfel'd doubles we will use an analog of the sum of
fundamental weights, commonly known as $\rho $. More precisely, we
define a character version of the linear form $(2\rho ,\cdot )$,
where $(\cdot ,\cdot )$ is the usual bilinear form on the
weight lattice.

Let $\ZIdual =\Hom (\ndZ ^I,\fienz )$
denote the group of characters of $\ndZ ^I$ with values in
$\fienz $.

\begin{defin}\label{de:rhomap}
  Let $\chi \in \cX $. Let $\rhomap \chi \in \widehat{\ndZ ^I}$
  such that
\[ \rhomap \chi (\al _i)=\chi (\al _i,\al _i) \quad \text{for all $i\in I$}.
\]
\end{defin}

\begin{lemma}
  Let $\chi \in \cX $, $p\in I$, and $\bnd =\bfun \chi (\al _p)$.
  Assume that
  $\bnd <\infty $. Then $\chi $ is $p$-finite and
  \[ \chi (\al _p,\beta )^{\bnd -1}\chi (\beta ,\al _p)^{\bnd -1}=
  \frac{\rhomap{r_p(\chi )}(\s _p^\chi (\beta ))}{\rhomap\chi (\beta )} \]
  for all $\beta \in \ndZ ^I$.
  \label{le:rho}
\end{lemma}

\begin{proof}
  Define $\xi _1,\xi _2:\ndZ ^I\to \fienz $ by
  \[ \xi _1(\beta )=\chi (\al _p,\beta )^{\bnd -1}\chi (\beta ,\al
  _p)^{\bnd -1},
  \qquad
  \xi _2(\beta )=\frac{\rhomap{r_p(\chi )}(\s _p^\chi (\beta ))}{
  \rhomap\chi (\beta )}. \]
  Then $\xi _1,\xi _2\in \ZIdual$.
  Thus it suffices to prove that $\xi _1(\al _j)=\xi _2(\al _j)$
  for all $j\in I$.
  Let $q_{jk}=\chi (\al _j,\al _k)$
  for all $j,k\in I$. Then
  $q_{pp}^\bnd =1$ since
  $\qnum{\bnd }{q_{pp}}=0$. Moreover,
  \begin{align*}
    \rhomap{r_p(\chi )}(\al _j)=&r_p(\chi )(\al _j,\al _j)\\
    =&\chi (\al _j-c_{pj}^\chi \al _p,\al _j-c_{pj}^\chi \al _p)
    =q_{jj}(q_{pj}q_{jp})^{-c_{pj}^\chi }q_{pp}^{c_{pj}^\chi c_{pj}^\chi }
  \end{align*}
  for all $j\in I$.
  
  By assumption, $\chi $ is $p$-finite, and hence
  for all $j\in I\setminus \{p\}$ we have
  $q_{pp}^{c_{pj}^\chi }=q_{pj}q_{jp}$ or
  $q_{pp}^{1-c_{pj}^\chi }=1$.
  Let $j\in I$. If $q_{pp}^{c_{pj}^\chi }=q_{pj}q_{jp}$, then
  \begin{gather*}
    \rhomap{r_p(\chi )}(\al _j)=q_{jj},\\
    \xi _2(\al _j)=q_{jj}^{-1}\rhomap{r_p(\chi )}(\s _p^\chi (\al _j))=
    q_{jj}^{-1}\rhomap{r_p(\chi )}(\al _j-c_{pj}^\chi \al _p)
    =q_{pp}^{-c_{pj}^\chi},\\
    \xi _1(\al _j)=q_{pj}^{\bnd -1} q_{jp}^{\bnd -1}=q_{pp}^{(\bnd -1)c_{pj}^\chi }
    =q_{pp}^{-c_{pj}^\chi },
  \end{gather*}
  and hence $\xi _1(\al _j)=\xi _2(\al _j)$. Otherwise $\bnd =1-c_{pj}^\chi $,
  $q_{pp}^{c_{pj}^\chi }=q_{pp}$, and
  then
  \begin{gather*}
    \rhomap{r_p(\chi )}(\al _p)=q_{pp},\quad
    \rhomap{r_p(\chi )}(\al _j)=q_{jj}(q_{pj}q_{jp})^{-c_{pj}^\chi }q_{pp},\\
    \xi _2(\al _j)=q_{jj}^{-1}\rhomap{r_p(\chi )}(\s _p^\chi (\al _j))=
    q_{jj}^{-1}\rhomap{r_p(\chi )}(\al _j-c_{pj}^\chi \al _p)
    =(q_{pj}q_{jp})^{-c_{pj}^\chi },\\
    \xi _1(\al _j)=q_{pj}^{\bnd -1} q_{jp}^{\bnd -1}
    =(q_{pj}q_{jp})^{-c_{pj}^\chi }.
  \end{gather*}
  Hence $\xi _1(\al _j)=\xi _2(\al _j)$ also in this case.
  This proves the lemma.
\end{proof}

\section{Multiparameter Drinfel'd doubles}
\label{sec:DD}

In this paper we study Verma modules for a class of Hopf algebras
introduced in \cite{p-Heck07b}. This class
contains multiparameter quantizations of semisimple Lie algebras
and basic classical Lie superalgebras.
The precise definition is given in
Eq.~\eqref{eq:Uchi}. It uses the Drinfel'd
double construction and the theory of Nichols algebras.

The Drinfel'd double \cite[Sect.\,3.2]{b-Joseph}
can be defined via a skew-Hopf pairing of two
Hopf algebras or as the quotient of a free associative algebra by a
certain ideal, see also Rem.~\ref{re:ideal}.
The first approach is more technical, but also more powerful. We
present here the second definition. For proofs see \cite{p-Heck07b}.

Let $I$ be a non-empty finite set,
$\chi $ a bicharacter on $\ndZ ^I$ with values in $\fienz $, and
$q_{i j}=\chi (\al _i,\al _j)$ for all $i,j\in I$.
Let $\cU (\chi )$ be the unital associative $\fie $-algebra with generators
$K_i$, $K_i^{-1}$, $L_i$, $L_i^{-1}$, $E_i$, and $F_i$, where $i\in I$,
and defining relations
\begin{align}
  XY= YX \quad & \makebox[0pt][l]{for all $X,Y\in \{K_i,K_i^{-1},
  L_i,L_i^{-1}\,|\,i\in I\}$,}
  \label{eq:KLrel}\\
  K_iK_i^{-1}=&\,1, & L_iL_i^{-1}=&\,1,
  \label{eq:KKrel}
\\
  K_iE_jK_i^{-1}=&\,q_{ij}E_j, & L_iE_jL_i^{-1}=&\,q_{ji}^{-1}E_j,
  \label{eq:KErel}
\\
  K_iF_jK_i^{-1}=&\,q_{ij}^{-1}F_j, & L_iF_jL_i^{-1}=&\,q_{ji}F_j,
  \label{eq:KFrel}\\
  E_iF_j&\makebox[0pt][l]{$-F_jE_i=\delta _{i,j}(K_i-L_i)$,}
  \label{eq:EFrel}
\end{align}
where $i,j\in I$, and $\delta _{i,j}$ denotes Kronecker's $\delta $.
The algebra $\cU (\chi )$ can be given a Hopf algebra
structure in many different ways. We will use the unique Hopf algebra structure
determined by
\begin{gather}
  \left\{
  \begin{aligned}
    \coun (K_i)=&\,1,\quad \coun (E_i)=0, &
    \coun (L_i)=&\,1,\quad \coun (F_i)=0, \\
    \copr (K_i)=&\,K_i\otimes K_i,&
    \copr (L_i)=&\,L_i\otimes L_i,\\
    \copr (K_i^{-1})=&\,K_i^{-1}\otimes K_i^{-1},&
    \copr (L_i^{-1})=&\,L_i^{-1}\otimes L_i^{-1},\\
    \copr (E_i)=&\,E_i\otimes 1+K_i\otimes E_i,&
    \copr (F_i)=&\,1\otimes F_i+F_i\otimes L_i
  \end{aligned}
  \right.
  \label{eq:coprcU}
\end{gather}
for all $i\in I$.

Let $\cU ^{+0},\cU ^{-0}$, and $\Uz $ denote the commutative cocommutative
Hopf subalgebras of $\cU (\chi )$ generated by
$\{K_i,K_i^{-1}\,|\,i\in I\}$,
$\{L_i,L_i^{-1}\,|\,i\in I\}$, and
$\{K_i,K_i^{-1},L_i,L_i^{-1}\,|\,i\in I\}$, respectively.
They are isomorphic to the ring of Laurent polynomials in $|I|$, $|I|$, and
$2|I|$ variables, respectively, in the natural way.
For any $\al =\sum _{i\in I}m_i\al _i\in \ndZ ^I$ let $K_\al =\prod _{i\in I}
K_i^{m_i}$ and $L_\al =\prod _{i\in I}L_i^{m_i}$.

Let $\cU ^+(\chi )$, $\cV ^+(\chi )$, $\cU ^-(\chi )$, and $\cV ^-(\chi )$
denote the subalgebras of $\cU (\chi )$ generated by
$\{E_i\,|\,i\in I\}$, $\{E_i,K_i,K_i^{-1}\,|\,i\in I\}$,
$\{F_i\,|\,i\in I\}$, and $\{F_i,L_i,L_i^{-1}\,|\,i\in I\}$,
respectively. Then $\cV ^+(\chi )$ and $\cV ^-(\chi )$ are Hopf subalgebras of
$\cU (\chi )$.

The algebra $\cU (\chi )$ admits a unique $\ndZ ^I$-grading
\begin{equation}
\begin{gathered}
  \cU (\chi )=\oplus _{\beta \in \ndZ ^I}\cU (\chi )_\beta ,\\
  1\in \cU (\chi )_0,\quad
  \cU (\chi )_\beta \cU (\chi )_\gamma \subset
  \cU (\chi )_{\beta +\gamma } \quad
  \text{for all $\beta ,\gamma \in \ndZ ^I$,}
\end{gathered}
  \label{eq:Zngrading}
\end{equation}
such that $K_i,K_i^{-1},L_i,L_i^{-1}\in \cU (\chi )_0$,
$E_i\in \cU (\chi )_{\al _i}$, and
$F_i\in \cU (\chi )_{-\al _i}$ for all $i\in I$.
Let
\[ \ndN _0^I=\Big\{\sum _{i\in I}a_i\al _i\,|\,a_i\in \ndN _0\Big\}\subset
   \ndZ ^I,
\]
and for any subspace $\cU '\subset \cU (\chi )$ and any $\beta \in \ndZ ^I$
let $\cU '_\beta =\cU '\cap \cU (\chi )_\beta $. Then
\begin{align*}
  \cU ^+(\chi )=&\oplus _{\beta \in \ndN _0^I}\cU ^+(\chi )_\beta ,&
  \cU ^-(\chi )=&\oplus _{\beta \in \ndN _0^I}\cU ^-(\chi )_{-\beta }.
\end{align*}

For all $\beta \in \ndZ ^I$ let
\begin{align}
  |\beta |=\sum _{i\in I} a_i\in \ndZ ,\quad \text{ where }
  \beta =\sum _{i\in I}a_i\al _i.
  \label{eq:abs}
\end{align}
The decomposition
\begin{equation}
	\cU (\chi )=\oplus _{m\in \ndZ }\cU (\chi )_m,\quad
	\text{where}\quad
	\cU (\chi )_m=\oplus _{\beta :|\beta |=m}\cU (\chi )_\beta ,
  \label{eq:Zgrading}
\end{equation}
gives a $\ndZ $-grading of $\cU (\chi )$ called the \textit{standard
grading}.

\begin{propo}\label{pr:algiso}
  Let $\chi \in \cX $.
  
	(1) Let $\ula =(a_i\,|\,i\in I)\in (\fienz )^I$.
	Then there exists a unique algebra automorphism
	$\varphi _{\ula }$ of $\cU (\chi )$ such that
\begin{equation}
	\varphi _{\ula }(K_i)=K_i,\,\,
	\varphi _{\ula }(L_i)=L_i,\,\,
	\varphi _{\ula }(E_i)=a_iE_i,\,\,
	\varphi _{\ula }(F_i)=a_i^{-1}F_i.
  \label{eq:cUauto1}
\end{equation}


(2) There is a unique algebra antiautomorphism
$\aaaU $ of $\cU (\chi )$
such that
\begin{align}
	\aaaU (K_i)=&K_i,& \aaaU (L_i)=&L_i,&
	\aaaU (E_i)=&F_i,& \aaaU (F_i)=&E_i.
	\label{eq:cUantiauto}
\end{align}
It satisfies the relation $\aaaU ^2=\id $.
\end{propo}

\begin{lemma}\label{le:commEFi}
	For all $i\in I$ there exist unique linear maps
	$\derK _i,\derL _i\in \End _\fie (\cU ^+(\chi ))$ such that
	\begin{align*}
		[E,F_i]=\derK _i(E)K_i-L_i\derL _i(E)\quad
		\text{for all $E\in \cU ^+(\chi )$.}
	\end{align*}
	The maps $\derK _i,\derL _i\in \End _\fie (\cU ^+(\chi ))$ are
	skew-derivations.
	More precisely,
	\begin{gather}
	  \derK _i(1)=\derL _i(1)=0,\quad 
	  \derK _i(E_j)=\derL _i(E_j)=\delta _{i,j},\label{eq:derKL1}\\
	\begin{aligned}
		\derK _i(EE')=&\derK _i(E)(K_i\lact E')+E\derK _i(E'),\\
		\derL _i(EE')=&\derL _i(E)E'+(L_i^{-1}\lact E)\derL _i(E')
	\end{aligned}
		\label{eq:derKL2}
	\end{gather}
	for all $i,j\in I$ and $E,E'\in \cU ^+(\chi )$.
\end{lemma}

Let $\cI ^+(\chi )$ be the unique maximal ideal of $\cU ^+(\chi )$ such that
$\cI ^+(\chi )\subset \ker \varepsilon $
and $\derK _i(\cI ^+(\chi ))\subset \cI ^+(\chi )$ for all $i\in I$.
Equivalently,
$\cI ^+(\chi )$ is the unique maximal ideal of $\cU ^+(\chi )$ such that
$\cI ^+(\chi )\subset \varepsilon $ and $\derL _i(\cI ^+(\chi ))\subset \cI
^+(\chi )$ for all $i\in I$, see \cite[Prop.\,5.4]{p-Heck07b}.
Let $\cI ^-(\chi )=\aaaU (\cI ^+(\chi ))$.
Let
\begin{align*}
  U^+(\chi )=&\cU ^+(\chi )/\cI ^+(\chi ), &
  U^-(\chi )=&\cU ^-(\chi )/\cI ^-(\chi ),\\
  V^+(\chi )=&\cV ^+(\chi )/\cI ^+(\chi )\cU ^{+0},&
  V^-(\chi )=&\cV ^-(\chi )/\cI ^-(\chi )\cU ^{-0},
\end{align*}
and
\begin{align}
  U(\chi )=\cU (\chi )/(\cI ^+(\chi ),\cI ^-(\chi )).
  \label{eq:Uchi}
\end{align}
The canonical inclusions
$\cU ^\pm (\chi )\subset \cU (\chi )$, $\Uz \subset \cU (\chi )$
induce maps
\[ \iota _+:U^+(\chi )\to U(\chi ),\quad
  \iota _0:\Uz \to U(\chi ),\quad
  \iota _-:U^-(\chi )\to U(\chi ).
\]

\begin{remar}\label{re:ideal}
  (i) The vector space $V=\oplus _{i\in I}\fie E_i$ is a \YD module
  over the group algebra
  $\fie \ndZ ^I\simeq \fie [K_i,K_i^{-1}\,|\,i\in I]\subset \Uz $,
  where the left action $\actl :\fie \ndZ ^I\otimes V\to V$
  and the left coaction $\coal :V\to \fie \ndZ ^I\ot V$
  are defined by
  \begin{align*}
    K_i \actl E_j=q_{ij}E_j,\qquad \coal (E_i)=K_i\ot E_i
  \end{align*}
  for all $i,j\in I$.
  The algebra $U^+(\chi )$ is commonly known as the
  \textit{Nichols algebra} of the \YD module $V$.

  (ii) There are various descriptions of the ideal $\cI ^+(\chi )$,
  see \textit{e.\,g.} \cite{inp-AndrSchn02}. In case of
  quantized enveloping algebras, see Sect.~\ref{sec:Uqg}, Serre
  relations generate the ideal $\cI ^+(\chi )$. A more general case is
  studied by Angiono \cite{p-Angi08}. For quantized Lie superalgebras
  the defining relations are determined in \cite{a-Yam99,a-Yam99e}.
  It is in general an open problem to give a nice set of generators of
  $\cI ^+(\chi )$, see \cite[Question\,5.9]{inp-Andr02}.
\end{remar}

\begin{propo} (Triangular decomposition)
  The map
  \[ \mathrm{m}(\iota _-\ot \iota _0\ot \iota _+):
  U^-(\chi )\ot \Uz \ot U^+(\chi )\to U(\chi ) \]
  is an isomorphism of $\ndZ ^I$-graded vector spaces, where
  $\mathrm{m}$ denotes the multiplication map.
  \label{pr:tridec}
\end{propo}

Following the convention in \cite[Sect.\,3.2.1]{b-Joseph},
a skew-Hopf pairing
$\sHp :A\times B\to \fie $, $(x,y)\mapsto \sHp (x,y)$,
of two Hopf algebras
$A$, $B$ is a bilinear map satisfying the equations
\begin{align}
  \label{eq:sHp1}
  \sHp (1,y)=&\,\coun (y),& \sHp (x,1)=&\,\coun (x),\\
  \label{eq:sHp2}
  \sHp (xx',y)=&\,\sHp (x',y_{(1)})\sHp (x,y_{(2)}),&
  \sHp (x,yy')=&\,\sHp (x_{(1)},y)\sHp (x_{(2)},y'),\\
  \label{eq:sHp3}
  &\qquad 
  \makebox[0pt][l]{$\sHp (\antip (x),y)=\sHp (x,\antip ^{-1}(y))$}
\end{align}
for all $x,x'\in A$ and $y,y'\in B$.

\begin{propo}
  \label{pr:sHpdef}
  (i) There exists a unique skew-Hopf pairing $\sHp $ of $\cV ^+(\chi )$ and
  $\cV ^-(\chi )$ such that for all $i,j\in I$ one has
  \begin{align*}
    \sHp (E_i,F_j)=-\delta _{i,j},\quad
    \sHp (E_i,L_j)=0,\quad
    \sHp (K_i,F_j)=0,\quad
    \sHp (K_i,L_j)=q_{ij}.
  \end{align*}

  (ii) The skew-Hopf pairing $\sHp $ satisfies the equations
  \begin{align*}
    \sHp (EK,FL)=\sHp (E,F)\sHp (K,L)
  \end{align*}
  for all $E\in \cU ^+(\chi )$,
  $F\in \cU ^-(\chi )$, $K\in \cU ^{+0}$, and $L\in \cU ^{-0}$.

  (iii) If $\beta ,\gamma \in \ndN _0^I$ with $\beta \not=\gamma $,
  $E\in \cU ^+(\chi )_\beta $, $F\in \cU ^-(\chi )_{-\gamma }$,
  then $\sHp (E,F)=0$.

  (iv) The restriction of $\sHp $ to
  $\cU ^+(\chi )\times \cU ^-(\chi )$
  induces a non-degenerate pairing
  $\sHp :U ^+(\chi )\times U^-(\chi )\to \fie $.
\end{propo}

\begin{proof}
  (i) and (ii) are \cite[Prop.\,4.3]{p-Heck07b}. (iii) follows from the
  definition of $\sHp $ and since $\copr $ is a $\ndZ ^I$-homogeneous
  map. (iv) was proven in \cite[Thm.\,5.8]{p-Heck07b}.
\end{proof}

By the general theory, see \cite[3.2.2]{b-Joseph},
the pairing $\sHp $ in Prop.~\ref{pr:sHpdef}
can be used to describe commutation
rules in $\cU (\chi )$ and $U(\chi )$. Namely,
\begin{align}
  \label{eq:Ucomm1}
  yx=&\sHp (x\_1,S(y\_1))x\_2y\_2\sHp (x\_3,y\_3),\\
  \label{eq:Ucomm2}
  xy=&\sHp (x\_1,y\_1)y\_2x\_2\sHp (x\_3,S(y\_3))
\end{align}
for all $x\in \cV ^+(\chi )$ and $y\in \cV ^-(\chi )$. Note that the
second formula follows from the first one and
Eqs.~\eqref{eq:sHp1}--\eqref{eq:sHp2}. 

Later we will also need some other general facts about $U(\chi )$.
Some of them are collected here. Let
\begin{align}
  \label{eq:kerderK}
  U^+_{i,K}(\chi )=&\ker \derK _i\subset U^+(\chi ),&
  U^+_{i,L}(\chi )=&\ker \derL _i\subset U^+(\chi ),\\
  \label{eq:kerderL}
  U^-_{i,K}(\chi )=&\aaaU (U^+_{i,K}),&
  U^-_{i,L}(\chi )=&\aaaU (U^+_{i,L}).
\end{align}
Recall the definition of $\bfun \chi $ in Eq.~\eqref{eq:height}.

\begin{lemma}
  Let $i\in I$.
  
  (i) Let $m\in \ndN $. The following are equivalent.
  \begin{itemize}
    \item $E_i^m=0$ in $U(\chi )$,
    \item $F_i^m=0$ in $U(\chi )$,
    \item $m\ge \bfun \chi (\al _i)$.
  \end{itemize}

  (ii) Let $\fie [E_i]$ and $\fie [F_i]$ be the subalgebras of
  $U(\chi )$
  generated by $E_i$ and $F_i$, respectively.
  The multiplication maps
  \begin{align*}
    U^+_{i,K}(\chi ) \ot \fie [E_i] \to &\,U^+(\chi ),&
    \fie [E_i] \ot U^+_{i,K}(\chi ) \to &\,U^+(\chi ),\\
    U^+_{i,L}(\chi ) \ot \fie [E_i] \to &\,U^+(\chi ),&
    \fie [E_i] \ot U^+_{i,L}(\chi ) \to &\,U^+(\chi ),\\
    U^-_{i,K}(\chi ) \ot \fie [F_i] \to &\,U^-(\chi ),&
    \fie [F_i] \ot U^-_{i,K}(\chi ) \to &\,U^-(\chi ),\\
    U^-_{i,L}(\chi ) \ot \fie [F_i] \to &\,U^-(\chi ),&
    \fie [F_i] \ot U^-_{i,L}(\chi ) \to &\,U^-(\chi ),
  \end{align*}
  are isomorphisms of $\ndZ ^I$-graded algebras.
  \label{le:Eheight}
\end{lemma}

\begin{proof}
  (i) is standard in the theory of Nichols algebras. It follows from
  Eqs.~\eqref{eq:derKL1}, \eqref{eq:derKL2} and the definitions
  of $\cI ^+(\chi )$ and $\cI ^-(\chi )$.
  The proof of (ii) for $U^+(\chi )$ can be performed as in
  \cite{a-Heck06a}.
  The formulas with $U^-(\chi )$ follow from those with
  $U^+(\chi )$ and Eqs.~\eqref{eq:kerderK}, \eqref{eq:kerderL}.
\end{proof}

\begin{lemma}
  \label{le:EmFn}
  Let $m,n\in \ndN _0$ and $p\in I$. Then
  \begin{align*}
    E_p^m F_p^n=\sum _{i=0}^{\makebox[0pt]{\scriptsize $\min \{m,n\}$}}
    {\textstyle
    \frac{\qfact{m}{q_{p p}}\qfact{n}{q_{p p}}}{
    \qfact{i}{q_{p p}}\qfact{m-i}{q_{p p}}\qfact{n-i}{q_{p p}}}
    }
    F_p^{n-i}\prod _{j=1}^i
    (q_{p p}^{i+j-m-n}K_p-L_p)E_p^{m-i}.
  \end{align*}
\end{lemma}

\begin{proof}
  For $n=0$ the claim is trivial.
  By \cite[Cor.~5.4]{p-Heck07b},
  \[ E_p^m F_p-F_p E_p^m=
  \qnum{m}{q_{p p}}(q_{p p}^{1-m}K_p-L_p)E_p^{m-1}. \]
  Hence the lemma holds for $n=1$.
  It suffices to check the claim for $m\ge n$, since then it also holds for
  $m<n$ using the algebra antiisomorphism $\aaaU $. 
  The proof of the lemma for $m\ge n$ is a standard calculation by
  induction on $n$.
\end{proof}

\section{An analogue of Lusztig's PBW basis}
\label{sec:Lusztig}

Let $\chi \in \cX $ and $p\in I$. Assume that $\chi $ is $p$-finite.
Let $q_{i j}=\chi (\al _i,\al _j)$ and $c_{p i}=c_{p i}^\chi $
for all $i,j\in I$.

For all $m\in \ndN _0$ and $i\in I\setminus \{p\}$ define recursively
$E^\pm _{i,m}\in U^+_{\al _i+m\al _p}$,
$F^\pm _{i,m}\in U^-_{\al _i+m\al _p}$ by
\begin{align*}
  E^+_{i,0}=&E_i, &
  E^+_{i,m+1}=&\,E_p E^+_{i,m}-(K_p\lact E^+_{i,m})E_p,\\
  E^-_{i,0}=&E_i, &
  E^-_{i,m+1}=&\,E_p E^-_{i,m}-(L_p\lact E^-_{i,m})E_p,\\
  F^+_{i,0}=&F_i, &
  F^+_{i,m+1}=&\,F_p F^+_{i,m}-(L_p\lact F^+_{i,m})F_p,\\
  F^-_{i,0}=&F_i, &
  F^-_{i,m+1}=&\,F_p F^-_{i,m}-(K_p\lact F^-_{i,m})F_p.
\end{align*}
We also define $E^+_{i,-1}=E^-_{i,-1}=F^+_{i,-1}=F^-_{i,-1}=0$.
The above definitions depend essentially on $p$.
If we want to emphasize this, we will write
$E^\pm _{i,m(p)}$ and $F^\pm _{i,m(p)}$
instead of $E^\pm _{i,m}$ and $F^\pm _{i,m}$, respectively.

For all $i\in I\setminus \{p\}$ define
\[ \lambda _i^\chi =\qfact{-c _{p i}}{q_{p p}}
\prod _{j=0}^{-c _{p i}-1}(q_{p p}^j q_{p i}q_{i p}-1). \]
Then $\lambda _i^\chi \not=0$ by definition of
$c_{p i}=c^\chi _{p i}$.

The next theorem was proven in \cite[Thm.\,6.11]{p-Heck07b}.

\begin{theor}\label{th:Liso}
  Let $\chi \in \cX $ and $p\in I$. Assume that $\chi $ is $p$-finite.
  Let $c_{pi}=c_{pi}^\chi $ for all $i\in I$.

  (i) There exist unique algebra isomorphisms
  $\LT _p, \LT _p^-: U (\chi )\to U (r_p(\chi ))$
  such that
  \begin{align*}
	\LT _p(K_p)=&\LT _p^-(K_p)=K _p^{-1},&
	\LT _p(K_i)=&\LT _p^-(K_i)=K _iK _p^{-c_{p i}},\\
	\LT _p(L_p)=&\LT _p^-(L_p)=L _p^{-1},&
	\LT _p(L_i)=&\LT _p^-(L_i)=L _iL _p^{-c_{p i}},\\
	\LT _p(E_p)=&F _p L _p^{-1},&
	\LT _p(E_i)=&E ^+_{i,-c_{p i}},\\
	\LT _p(F_p)=&K _p^{-1}E _p,&
	\LT _p(F_i)=&\lambda _i(r_p(\chi ))^{-1}F ^+_{i,-c_{p i}},\\
	\LT _p^-(E_p)=&K _p^{-1}F _p,&
	\LT _p^-(E_i)=&\lambda _i(r_p(\chi ^{-1}))^{-1} E ^-_{i,-c_{pi}},\\
	\LT _p^-(F_p)=&E _p L _p^{-1},&
	\LT _p^-(F_i)=&(-1)^{c_{p i}} F ^-_{i,-c_{p i}}.
  \end{align*}

  (ii) The maps $\LT _p$, $\LT _p^-$ satisfy
  $\LT _p \LT _p^-=\LT _p^-\LT _p=\id _{U(\chi )}$.

  (iii) There exists a unique $\ula \in (\fienz )^I$ such that
  $\LT _p \aaaU =\aaaU \LT ^-_p \varphi _{\ula }$ in $\Hom (U(\chi
  ),U(r_p(\chi )))$.
\end{theor}

Note that $\LT _p\LT _p^-$ is an automorphism of $U(\chi )$ if one regards
$\LT _p^-$ as a map from $U(\chi )$ to $U(r_p(\chi ))$ and
$\LT _p$ as a map from $U(r_p(\chi ))$ to $U(r_p r_p(\chi ))=U(\chi )$.

\begin{propo} \label{pr:LTdeg}
  Let $\chi \in \cX $ and $p\in I$. Assume that $\chi $ is $p$-finite.
  Then
  \[ \LT _p(U(\chi )_\al )=U(r_p(\chi ))_{\s _p^\chi (\al )} \quad
  \text{for all $\al \in \ndZ ^I$.} \]
\end{propo}

\begin{proof}
  The maps $\LT _p:U(\chi )\to U(r_p(\chi ))$ and $\LT _p^-:U(r_p(\chi ))\to
  U(\chi )$ are mutually inverse algebra isomorphisms, and send
  generators of degree $\al $ into the homogeneous component of degree
  $\s _p(\al )$.
\end{proof}

\begin{lemma}
  \label{le:TpU+U+}
  Let $\chi \in \cX $ and $p\in I$. Assume that $\chi $ is $p$-finite.
  Then
  \begin{align*}
    \LT _p(U^+_{p,L}(\chi ))=&\,U^+_{p,K}(r_p(\chi )), &
    \LT _p(U^-_{p,K}(\chi ))=&\,U^-_{p,L}(r_p(\chi )),\\
    \LT ^-_p(U^+_{p,K}(\chi ))=&\,U^+_{p,L}(r_p(\chi )), &
    \LT ^-_p(U^-_{p,L}(\chi ))=&\,U^-_{p,K}(r_p(\chi )).
  \end{align*}
\end{lemma}

\begin{proof}
  Since $\chi $ and $r_p(\chi )$ are $p$-finite,
  \cite[Prop.\,5.10]{p-Heck07b} and \cite[Prop.\,6.7(d)]{p-Heck07b}
  give that
  \[ \LT _p(U^+_{p,L}(\chi ))\subset U^+_{p,K}(r_p(\chi )),\qquad
  \LT _p^-(U^+_{p,K}(r_p(\chi )))\subset U^+_{p,L}(\chi ). \]
  Thus
  $\LT _p(U^+_{p,L}(\chi ))=U^+_{p,K}(r_p(\chi ))$ by
  Thm.~\ref{th:Liso}(ii). Similar arguments yield that
  $\LT ^-_p(U^+_{p,K}(\chi ))=U^+_{p,L}(r_p(\chi ))$.
  The remaining two equations can be obtained from these and
  Thm.~\ref{th:Liso}(iii).
\end{proof}

In the rest of the section assume that
$\chi \in \cX _3$. Let $n=|R_+^\chi |\in \ndN $.
The following construction generalizes the Poincar\'e-Birkhoff-Witt
basis of quantized enveloping algebras given by Lusztig.

Let $i_1,i_2,\dots ,i_n\in I$ such that
$\ell (1_{\chi }\s _{i_1}\s _{i_2}\cdots \s _{i_n})=n$.
For all $\nu \in \{1,2,\dots,n\}$ let
\begin{gather}\label{eq:betak}
  \beta _\nu ^\chi
  =1_\chi \s _{i_1}\s _{i_2}\dots \s _{i_{\mu -1}}(\al _{i_\nu }).
\end{gather}
Then the elements $\beta _\nu ^\chi $, $1\le \nu \le n$,
are pairwise different and
\begin{align}
  R^\chi _+=\{\beta _\nu ^\chi \,|\,1\le \nu \le n\}
  \label{eq:proots}
\end{align}
by \cite[Prop.\,2.12]{p-CH08}. For all $\nu \in \{1,2,\dots,n\}$ let
\begin{gather}
  \label{eq:Ebetak}
  E_{\beta _\nu }= E_{\beta _\nu }^\chi 
  =\LT _{i_1}\dots \LT _{i_{\nu -1}}(E_{i_\nu }),\quad
  F_{\beta _\nu }= F_{\beta _\nu }^\chi 
  =\LT _{i_1}\dots \LT _{i_{\nu -1}}(F_{i_\nu }),\\
  \label{eq:Ebarbetak}
  \bE _{\beta _\nu }= \bE _{\beta _\nu }^\chi 
  =\LT ^-_{i_1}\dots \LT ^-_{i_{\nu -1}}(E_{i_\nu }),\quad
  \bF _{\beta _\nu }= \bF _{\beta _\nu }^\chi 
  =\LT ^-_{i_1}\dots \LT ^-_{i_{\nu -1}}(F_{i_\nu }),
\end{gather}
where $E_{i_\nu },F_{i_\nu }\in U(r_{i_{\nu -1}}\dots r_{i_2}r_{i_1}(\chi ))$.
Then 
\begin{align}
  E_{\beta _\nu },\bE _{\beta _\nu }\in U^+(\chi )_{\beta _\nu },\qquad
  F_{\beta _\nu },\bF _{\beta _\nu }\in U^-(\chi )_{-\beta _\nu }
  \label{eq:EbetainU+}
\end{align}
for all $\nu \in \{1,\dots ,n\}$ by \cite[Thm.\,6.19]{p-Heck07b},
Thm.~\ref{th:Liso}(iii) and Prop.~\ref{pr:LTdeg}.

\begin{lemma}
  \label{le:rvrel}
  Assume that $\chi \in \cX _3$. Let $\nu \in \{1,2,\dots ,n\}$,
  and assume that
  $\bfun \chi (\beta _\nu )<\infty $.
  Then $E_{\beta _\nu }^{\bfun \chi (\beta _\nu )}=
  F_{\beta _\nu }^{\bfun \chi (\beta _\nu )}=0$ in $U(\chi )$.
\end{lemma}

\begin{proof}
  By Eq.~\eqref{eq:hghtrpchi} and since $T_i$ is an isomorphism for
  each $i\in I$,
  it suffices to prove that
  $E_i^{\bfun{\chi '}(\al _i)}=F_i^{\bfun{\chi '}(\al _i)}=0$ in $U(\chi ')$
  for all $\chi '\in \cX $ and $i\in I$ with $\bfun{\chi '}(\al _i)<\infty $.
  This follows from Lemma~\ref{le:Eheight}(i).
\end{proof}

\begin{theor}\label{th:PBW}
  Assume that $\chi \in \cX _3$. Let $n=|R_+^\chi |\in \ndN $.
  Both sets
  \begin{align}
    \label{eq:LusztigPBW+}
    \big\{ E_{\beta _1}^{m_1} E_{\beta _2}^{m_2}\cdots E_{\beta _n}^{m_n}\,&|\,
    0\le m_\nu <\bfun \chi (\beta _\nu )
    \text{ for all $\nu \in \{1,2,\dots ,n\}$} \big\},\\
    \label{eq:LusztigPBW-}
    \big\{ \bE _{\beta _1}^{m_1} \bE _{\beta _2}^{m_2}\cdots
    \bE _{\beta _n}^{m_n}\,&|\,
    0\le m_\nu <\bfun \chi (\beta _\nu )
    \text{ for all $\nu \in \{1,2,\dots ,n\}$} \big\}
  \end{align}
  form vector space bases of $U ^+(\chi )$.
\end{theor}

\begin{proof}
  We prove the claim for the basis in Eq.~\eqref{eq:LusztigPBW+}. For the
  other set the proof is analogous.
  By Eqs.~\eqref{eq:roots} and \eqref{eq:proots},
  \begin{align*}
    \dim U^+(\chi )_\al =\Big|\big\{(m_1,m_2,&\dots ,m_n)\in \ndN _0^n\,\big|\,
    \sum _{\nu =1}^n m_\nu \beta _\nu =\al ,\\
    &m_\nu <\bfun \chi (\beta _\nu ) \text{ for all $\nu \in \{1,\dots ,n\}$}
    \big\}\Big|
  \end{align*}
  for all $\al \in \ndN _0^I$. Since $\deg E_{\beta _\nu }=\beta _\nu $ for all
  $\nu \in \{1,2,\dots ,n\}$, it suffices to show that for all
  $\mu \in \{1,2,\dots ,n+1\}$ the elements of the set
  \begin{align*}
    \big\{ E_{\beta _\mu }^{m_\mu } E_{\beta _{\mu +1}}^{m_{\mu +1}}\cdots
    E_{\beta _n}^{m_n}\,|\, 0\le m_\nu <\bfun \chi (\beta _\nu )
    \text{ for all $\nu \in \{\mu ,\mu +1 ,\dots ,n\}$} \big\}
  \end{align*}
  are linearly independent. We proceed by induction on $n+1-\mu $. If
  $\mu =n+1$,
  then the above set is empty, and hence its elements are linearly independent.

  Let now $\mu \in \{1,2,\dots ,n\}$.
  For all $m_\mu ,\dots ,m_n\in \ndN _0$ with $m_\nu <\bfun \chi (\beta _\nu )$
  for all $\nu \in \{\mu ,\mu +1,\dots ,n\}$
  let $a_{m_\mu ,\dots ,m_n}\in \fie $. Assume that
  \begin{align}
    \sum _{m_\mu ,\dots ,m_n}a _{m_\mu ,\dots ,m_n}
    E_{\beta _\mu }^{m_\mu } E_{\beta _{\mu +1}}^{m_{\mu +1}}\cdots
    E_{\beta _n}^{m_n}=0
    \label{eq:lindep}
  \end{align}
  in $U^+(\chi )$. Let
  $\LT ^-=\LT ^-_{i_\mu }\cdots \LT ^-_{i_2}\LT ^-_{i_1}$.
  Since $\LT ^-(E_{\beta _\mu })=\LT ^-_{i_\mu }(E_{i_\mu })
  =K_{i_\mu }^{-1}F_{i_\mu }$,
  we obtain that
  \[ \sum _{m_\mu ,\dots ,m_n}
  a _{m_\mu ,\dots ,m_n}
  (K_{i_\mu }^{-1}F_{i_\mu })^{m_\mu }\LT ^-(E_{\beta _{\mu +1}})^{m_{\mu +1}}
  \cdots \LT ^-(E_{\beta _n})^{m_n}=0. \]
  Since
  $\LT ^-(E_{\beta _\nu })\in U^+(r_{i_\mu }\cdots r_{i_2}r_{i_1}(\chi ))$
  for all $\nu \in \{\mu +1,\mu +2,\dots ,n\}$,
  Prop.~\ref{pr:tridec} implies that
  \[ \sum _{m_{\mu +1},\dots ,m_n}a _{m_\mu ,m_{\mu +1},\dots ,m_n}
  \LT ^-(E_{\beta _{\mu +1}})^{m_{\mu +1}}\cdots \LT ^-(E_{\beta _n})^{m_n}
  =0 \]
  for all $m_\mu \in \ndN _0$, $m_\mu <\bfun \chi (\beta _\mu )$.
  Therefore
  \[ \sum _{m_{\mu +1},\dots ,m_n}a _{m_\mu ,m_{\mu +1},\dots ,m_n}
  E_{\beta _{\mu +1}}^{m_{\mu +1}}\cdots E_{\beta _n}^{m_n}=0 \]
  for all $m_\mu \in \ndN _0$, $m_\mu <\bfun \chi (\beta _\mu )$.
  Then $a_{m_\mu ,m_{\mu +1},\dots ,m_n}=0$
  for all $(m_\mu ,\dots ,m_n)$ by induction
  hypothesis, which proves the induction step. Thus the theorem holds.
\end{proof}

\begin{lemma}
  Assume that $\chi \in \cX _3$. Then
  $\ker (\derK _{i_1}:U^+(\chi )\to U^+(\chi ))$ coincides with
  the subalgebra of $U^+(\chi )$
  generated by the elements $E_{\beta _\nu }$, $\nu \in \{2,3,\dots ,n\}$.
  The set
  \begin{align}
    \big\{ E_{\beta _2}^{m_2} E_{\beta _3}^{m_3}\cdots E_{\beta _n}^{m_n}\,&|\,
    0\le m_\nu <\bfun \chi (\beta _\nu )
    \text{ for all $\nu \in \{2,3,\dots ,n\}$} \big\}
  \end{align}
  forms a vector space basis of $\ker \derK _{i_1}$.
  \label{le:kerderK}
\end{lemma}

\begin{proof}
  Let $\nu \in \{2,3,\dots ,n\}$. By \cite[Lemma\,5.10]{p-Heck07b} and
  Lemma~\ref{le:rvrel}
  there exist $m\in \ndN _0$ and $x_0,x_1,\dots ,x_m\in \ker \derK _{i_1}$
  such that $m<\bfun \chi (\al _{i_1})$ and
  $E_{\beta _\nu }=\sum _{\mu =0}^m x_\mu E_{i_1}^\mu $. Then
  \[ \LT _{i_1}^-(E_{\beta _\nu })=
  \sum _{\mu =0}^m\LT ^-_{i_1}(x_\mu ) (K_{i_1}^{-1}F_{i_1})^\mu . \]
  Moreover,
  $\LT ^-_{i_1}(x_\mu )\in U^+(r_{i_1}(\chi ))$ for all $\mu \in \{0,1,\dots
  ,m\}$ by \cite[Prop.\,5.19,\,Lemma\,6.7(d)]{p-Heck07b}.
  Since $\LT ^-_{i_1}(E_{\beta _\nu })\in U^+(r_{i_1}(\chi ))$,
  triangular decomposition of $U(r_{i_1}(\chi ))$ implies that $x_\mu =0$ for
  all $\mu >0$. Hence $E_{\beta _\nu }=x_0\in \ker \derK _{i_1}$.
  Then the claim of the lemma follows from the inclusions
  \[ \langle E_{\beta _\kappa }\,|\,\kappa \in \{2,3,\dots ,n\}\rangle
  \subset \ker \derK _{i_1} \subset
  \mathop{\oplus }_{ {(m_2,m_3,\dots ,m_n) \atop
  0\le m_\kappa <\bfun \chi (\beta _\kappa )\,\text{for all $\kappa $}}}
  \fie E_{\beta _2}^{m_2}\cdots E_{\beta _n}^{m_n},
  \]
  where the second inclusion is obtained from
  Thm.~\ref{th:PBW} and the formula
  \[ \derK _{i_1}( E_{\beta _1}^{m_1}E_{\beta _2}^{m_2}\cdots
  E_{\beta _n}^{m_n})= \qnum{m_1}{q_{i_1i_1}}
  E_{\beta _1}^{m_1-1}K_{i_1}\actl (E_{\beta _2}^{m_2}\cdots
  E_{\beta _n}^{m_n}).\]
\end{proof}

\begin{remar}
  The analogous version of Lemma~\ref{le:kerderK} for $\ker \derL _{i_1}$
  is obtained by replacing $E_{\beta _\nu }$ by $\bE _{\beta _\nu }$
  for all $\nu \in \{2,3,\dots ,n\}$.
\end{remar}

\begin{theor}\label{th:EErel}
  Assume that $\chi \in \cX _3$. Let $n=|R_+^\chi |\in \ndN $.
  Then
  \begin{align*}
    E_{\beta _\mu }E_{\beta _\nu }-\chi (\beta _\mu ,\beta _\nu )
    E_{\beta _\nu }E_{\beta _\mu }
    \in \, & \langle E_{\beta _\kappa }\,|\,\mu <\kappa <\nu \rangle
    \subset U^+(\chi ),\\
    \bE _{\beta _\mu }\bE _{\beta _\nu }
    -\chi ^{-1}(\beta _\nu ,\beta _\mu )\bE _{\beta _\nu }\bE _{\beta _\mu }
    \in \, & \langle \bE _{\beta _\kappa }
    \,|\,\mu <\kappa <\nu \rangle \subset U^+(\chi ),\\
    F_{\beta _\mu }F_{\beta _\nu }-\chi (\beta _\nu ,\beta _\mu )
    F_{\beta _\nu }F_{\beta _\mu }
    \in \, & \langle F_{\beta _\kappa }\,|\,\mu <\kappa <\nu \rangle
    \subset U^-(\chi ),\\
    \bF _{\beta _\mu }\bF _{\beta _\nu }-\chi ^{-1}(\beta _\mu ,\beta _\nu )
    \bF _{\beta _\nu }\bF _{\beta _\mu }
    \in \, & \langle \bF _{\beta _\kappa }\,|\,\mu <\kappa <\nu \rangle
    \subset U^-(\chi )
  \end{align*}
  for all $\mu ,\nu \in \{1,2,\dots ,n\}$ with $\mu <\nu $.
\end{theor}

\begin{proof}
  We prove the first relation for $\mu =1$ and all $\nu \in \{2,3,\dots ,n\}$.
  Then the first relation for $\mu >1$ follows from
  \[ E_{\beta _\mu }E_{\beta _\nu }-\chi (\beta _\mu ,\beta _\nu )
  E_{\beta _\nu }E_{\beta _\mu }
  =\LT _{i_1}\cdots \LT _{i_{\mu -1}}(
  E_{i_\mu }E'_\nu -\chi (\beta _\mu ,\beta _\nu )E'_\nu E_{i_\mu }),
  \]
  where $E'_\nu =\LT _{i_\mu }\cdots \LT _{i_{\nu -1}}(E_{i_\nu })$,
  by using the case $\mu =1$, Eq.~\eqref{eq:w*chi} and the first relation in
  \eqref{eq:EbetainU+}.
  The proof of the second relation of the theorem is similar.
  The third and fourth relations can be obtained from the first two by
  applying $\aaaU $ and using the formulas
  \begin{align}
    \aaaU (E_{\beta _\kappa })\in \fienz \bF _{\beta _\kappa },\quad
    \aaaU (\bE _{\beta _\kappa })\in \fienz F _{\beta _\kappa },
    \qquad \kappa \in \{1,2,\dots ,n\},
    \label{eq:aaaU(E)}
  \end{align}
  which follow from Thm.~\ref{th:Liso}(iii).

  Let $\nu \in \{2,3,\dots ,n\}$.
  For all $(m_1,m_2,\dots ,m_n)\in \ndN _0^I$ with
  $m_\kappa <\bfun \chi (\beta _\kappa )$ for all $\kappa \in \{1,2,\dots ,n\}$
  let $a_{m_1,\dots ,m_n}\in \fie $ such that
  \begin{align}
    E_{i_1}E_{\beta _\nu }-\chi (\al _{i_1},\beta _\nu )E_{\beta _\nu }E_{i_1}
    =\sum _{m_1,\dots ,m_n}a_{m_1,\dots ,m_n}E_{\beta _1}^{m_1}
    \cdots E_{\beta _n}^{m_n}.
    \label{eq:EE-EE}
  \end{align}
  The numbers $a_{m_1,\dots ,m_n}\in \fie $
  exist and are unique by Thm.~\ref{th:PBW}.
  Let $\chi _\nu =r_{i_\nu }\cdots r_{i_2}r_{i_1}(\chi )$.
  Apply to Eq.~\eqref{eq:EE-EE} the isomorphism
  $\LT ^-=\LT ^-_{i_\nu }\cdots \LT ^-_{i_2}\LT ^-_{i_1}\in \Hom (U(\chi ),
  U(\chi _\nu ))$.
  For all $\kappa \in \{1,2,\dots ,\nu \}$,
  \begin{align*}
    \LT ^-_{i_\nu }\cdots \LT ^-_{i_2}\LT ^-_{i_1}(E_{\beta _\kappa })
    =&\LT ^-_{i_\nu }\cdots \LT ^-_{i_{\kappa +1}}
    \LT ^-_{i_\kappa }(E_{i_\kappa })\\
    =&\LT ^-_{i_\nu }\cdots \LT ^-_{i_{\kappa +1}}
    (K_{i_\kappa }^{-1}F_{i_\kappa })
    \in U^-(\chi _\nu )\Uz &
  \end{align*}
  by Eq.~\eqref{eq:EbetainU+}. Hence
  \begin{align*}
    &\sum _{m_1,\dots ,m_n}a_{m_1,\dots ,m_n}
    \LT ^-(E_{\beta _1}^{m_1}\cdots E_{\beta _\nu }^{m_\nu })
    \LT ^-(E_{\beta _{\nu +1}}^{m_{\nu +1}}\cdots E_{\beta _n}^{m_n})\\
    &\quad =\LT ^-(E_{i_1}E_{\beta _\nu }
    -\chi (\al _{i_1},\beta _\nu )E_{\beta _\nu }E_{i_1})\in U^-(\chi _\nu )\Uz .
  \end{align*}
  By triangular decomposition of $U(\chi )$ it follows that
  $a_{m_1,\dots ,m_n}=0$ for all $(m_1,\dots ,m_n)$ with $m_\kappa >0$
  for some $\kappa \in \{\nu +1,\nu +2,\dots ,n\}$.

  By Lemma~\ref{le:kerderK}, $E_{\beta _\nu }\in \ker \derK _{i_1}$. Hence
  $E_{i_1}E_{\beta _\nu }-\chi (\al _{i_1},\beta _\nu )E_{\beta _\nu }E_{i_1}
  \in \ker \derK _{i_1}$ by Lemma~\ref{le:commEFi}. Thus Lemma~\ref{le:kerderK}
  implies that $a_{m_1,\dots ,m_n}=0$ whenever $m_1>0$.

  Suppose that there exists $(m_1,\dots ,m_n)$ with $m_\nu >0$ and
  $a_{m_1,\dots ,m_n}\not=0$. Since 
  $E_{i_1}E_{\beta _\nu }-\chi (\al _{i_1},\beta _\nu )E_{\beta _\nu }E_{i_1}$
  is $\ndZ ^I$-homogeneous of degree $\al _{i_1}+\beta _\nu $,
  the only possibility is that $m_1=1$, $m_\nu =1$, and $m_\kappa =0$
  for all $\kappa \notin \{1,\nu \}$. Since $a_{m_1,\dots ,m_n}\not=0$,
  this is a contradiction to the previous paragraph. Thus the theorem is
  proven.
\end{proof}

Next we prove a generalization of Thm.~\ref{th:PBW}.

\begin{theor}\label{th:PBWtau}
  Assume that $\chi \in \cX _3$. Let $n=|R^\chi _+|$ and
  let $\tau $ be a permutation of the set $\{1,2,\dots ,n\}$.
  Then the sets
  \begin{align*}
    \big\{ E_{\beta _{\tau (1)}}^{m_{\tau (1)}} E_{\beta _{\tau (2)}}^{m_{\tau
    (2)}}\cdots E_{\beta _{\tau (n)}}^{m_{\tau (n)}}\,&|\,
    0\le m_\nu <\bfun \chi (\beta _\nu )
    \text{ for all $\nu \in \{1,2,\dots ,n\}$} \big\},\\
    \big\{ \bE _{\beta _{\tau (1)}}^{m_{\tau (1)}}
    \bE _{\beta _{\tau (2)}}^{m_{\tau (2)}}\cdots 
    \bE _{\beta _{\tau (n)}}^{m_{\tau (n)}}\,&|\,
    0\le m_\nu <\bfun \chi (\beta _\nu )
    \text{ for all $\nu \in \{1,2,\dots ,n\}$} \big\}
  \end{align*}
  form vector space bases of $U ^+(\chi )$, and the sets
  \begin{align*}
    \big\{ F_{\beta _{\tau (1)}}^{m_{\tau (1)}} F_{\beta _{\tau (2)}}^{m_{\tau
    (2)}}\cdots F_{\beta _{\tau (n)}}^{m_{\tau (n)}}\,&|\,
    0\le m_\nu <\bfun \chi (\beta _\nu )
    \text{ for all $\nu \in \{1,2,\dots ,n\}$} \big\},\\
    \big\{ \bF _{\beta _{\tau (1)}}^{m_{\tau (1)}}
    \bF _{\beta _{\tau (2)}}^{m_{\tau (2)}}\cdots 
    \bF _{\beta _{\tau (n)}}^{m_{\tau (n)}}\,&|\,
    0\le m_\nu <\bfun \chi (\beta _\nu )
    \text{ for all $\nu \in \{1,2,\dots ,n\}$} \big\}
  \end{align*}
  form vector space bases of $U ^-(\chi )$.
\end{theor}

\begin{proof}
  It suffices to prove that the first set is a basis of $U^+(\chi )$.
  Indeed, the proof for the second set can be obtained by using the maps
  $\LT ^-_i$, where $i\in I$, instead of $\LT _i$.
  The second part of the claim follows from the first part by applying the
  algebra antiautomorphism $\aaaU $ and using Eq.~\eqref{eq:aaaU(E)}.

  For any $\ulm =(m_1,\ldots ,m_n)\in \ndN _0^n$ let $|\ulm |=\sum _{\mu =1}^n
  m_\mu |\beta _\mu |$, where $|\al |=\sum _{j\in I}a_j$
  for all $\al =\sum _{j\in I}a_j\al _j \in \ndN _0^I$.
  Let $N$ be the (additive) monoid $\ndN _0^n$ equipped with the following ordering:
  \[ \ulm '<\ulm \quad \Leftrightarrow \quad |\ulm '|<|\ulm | \quad
  \text{or} \quad |\ulm '|=|\ulm |,\,\ulm '<_{\mathrm{lex}}\ulm ,\]
  where $<_{\mathrm{lex}}$ means lexicographical ordering. We use the
  convention $\ulm <_{\mathrm{lex}}\ulm $.
  The ordering $<$ is a total ordering.
  For all $\ulm \in \ndN _0^n$ define
  \[ \cF ^{\ulm }U^+(\chi )=\mathop{\oplus }_{\ulm '\in N,\ulm '<\ulm}
  \fie E_{\beta _1}^{m'_1}
  E_{\beta _2}^{m'_2}\cdots E_{\beta _n}^{m'_n} \subset U^+(\chi ).\]
  The vector spaces $\cF ^{\ulm }U^+(\chi )$, where $\ulm \in N$, are
  finite-dimensional,
  since the degrees of their elements are bounded.
  Moreover,
  \[ \cF ^0 U^+(\chi )=\fie 1,\qquad
  \cF ^{\ulm }U^+(\chi ) \cF ^{\ulm '}U^+(\chi )\subset
  \cF ^{\ulm +\ulm '}U^+(\chi ) \]
  for all $\ulm ,\ulm '\in N$ by Thm.~\ref{th:EErel} and since $U^+(\chi )$ is
  $\ndZ ^I$-graded. Thus $\cF $ defines a filtration of $U^+(\chi )$ by the
  monoid $N$, and the corresponding graded algebra
  \[ \mathop{\oplus }_{\ulm \in N} \Big(\cF ^{\ulm }U^+(\chi )
  / \sum _{\ulm '<\ulm ,\ulm
  '\not=\ulm } \cF ^{\ulm '}U^+(\chi )\Big) \]
  is a skew-polynomial ring in $n$ variables by Thm.~\ref{th:EErel}.
  By a standard conclusion we
  obtain that the first set in the claim of the theorem is indeed a basis of
  $U^+(\chi )$.
\end{proof}

\section{Verma modules and morphisms}
\label{sec:Verma}

We consider Verma modules for the algebras $U(\chi )$, $\chi \in \cX $.
We observe that the fundaments of the theory
of Verma modules for quantized
enveloping algebras can be carried over to a great extent to $U(\chi )$.
New phenomena appear if some generators of $U(\chi )$ are nilpotent.

Let $\fiee $ be a field extension of $\fie $.
Let $\charUz $ denote the set of $\fieenz $-valued characters
(algebra maps from $\Uz $ to $\fieenz $)
of the group algebra $\Uz $. For all $\chi \in \cX $
there is a natural group homomorphism
\begin{align}
  \ZIch \chi : \ndZ ^I\to \charUz ,\quad
  \ZIch \chi (\al )(K_\beta L_{\beta '})=
  \chi (\beta ,\al )\chi (\al,\beta ')^{-1}
  \label{eq:ZIch}
\end{align}
for all $\al ,\beta ,\beta ' \in \ndZ ^I$.
If $\chi \in \cX $ and $p\in I$ such that $\chi $ is $p$-finite, then
\begin{align}
  \ZIch{r_p(\chi )}(\s _p^\chi (\al ))
  (K_{\s _p^\chi (\beta )}L_{\s _p^\chi (\beta ')})=
  \ZIch \chi (\al )(K_\beta L_{\beta '})
  \label{eq:ZIchrefl}
\end{align}
by Eq.~\eqref{eq:w*chi}.

Let $\chi \in \cX $.
Given a $\fieenz $-valued character $\Lambda \in \charUz $, one can regard
$\fiee $ as a one-dimensional $U ^+(\chi )\Uz $-module with generator
$1_\Lambda =1$
via
\begin{align}
  uE 1_\Lambda =\coun (E)\Lambda (u)1_\Lambda \qquad
  \text{for all $E\in U ^+(\chi )$, $u\in \Uz $.}
  \label{eq:fieLambda}
\end{align}
We write $\fiee _\Lambda $ for this module.

\begin{defin}
  A \textit{Verma module} of $U (\chi )$ is a $U (\chi )$-module
  of the form
  \[ M^\chi (\Lambda )=U (\chi )\ot _{U ^+(\chi )\Uz }\fiee _\Lambda ,\quad
  \text{where $\Lambda \in \charUz $.} \]
  We write $v^\chi _\Lambda $ or just $v_\Lambda $ for
  $1\ot 1_\Lambda \in M^\chi (\Lambda )$.
\end{defin}

Any Verma module is also a $\fiee $-module via the $\fiee $-module
structure of the second
tensor factor.

Let $\Lambda \in \charUz $. Triangular decomposition of $U (\chi )$ gives
the following standard fact.

\begin{lemma}
  The map $U ^-(\chi )\ot _\fie \fiee \to M^\chi (\Lambda )$,
  $u\ot x\mapsto uxv_\Lambda =u\ot x1_\Lambda $,
  is an isomorphism of vector spaces over $\fiee $.
  \label{le:MLiso}
\end{lemma}

The isomorphism in Lemma~\ref{le:MLiso}
and the $\ndZ ^I$-grading of $U^-(\chi )$
induce a unique $\ndZ ^I$-grading on $M^\chi (\Lambda )$ such that
\begin{align}
  M^\chi (\Lambda )_\al =U^-(\chi )_\al \ot _\fie \fiee _\Lambda
  \quad \text{for all $\al \in \ndZ ^I$.}
  \label{eq:MLgrading}
\end{align}
Then
\begin{align}
  M^\chi (\Lambda )_0=\fiee v_\Lambda ,\,\,
  U(\chi )_\al M^\chi (\Lambda )_\beta \subset M^\chi (\Lambda )_{\al +\beta }
  \,\,\text{for all $\al ,\beta \in \ndZ ^I$.}
  \label{eq:MLgrading2}
\end{align}
Moreover, $M^\chi (\Lambda )_\al \not=0$ implies that $-\al \in \ndN _0^I$.

The group algebra $\Uz $ acts on $M^\chi (\Lambda )$ via left multiplication.
This action is given by characters:
\begin{align}
  uv=(\Lambda +\ZIch \chi (\al ))(u)v \quad
  \text{for all $u\in \Uz $, $\al \in \ndZ ^I$,
  $v\in M^\chi (\Lambda )_\al $,}
  \label{eq:U0M}
\end{align}
see Eqs.~\eqref{eq:ZIch}, \eqref{eq:fieLambda}, \eqref{eq:KErel},
and \eqref{eq:KFrel}.

Let $\Lambda \in \charUz $.
The family of those $U(\chi )\ot _\fie \fiee $-submodules
of $M^\chi (\Lambda )$,
which are contained in
$\oplus _{\al \not=0}M^\chi (\Lambda )_\al $,
have a unique maximal element $I^\chi (\Lambda )$.
Let
\begin{align}
  L^\chi (\Lambda )=M^\chi (\Lambda )/I^\chi (\Lambda )
  \label{eq:LLambda}
\end{align}
be the quotient $U(\chi )$-module. The maximality of $I^\chi (\Lambda )$
implies that
\[ I^\chi (\Lambda )=\big(I^\chi (\Lambda )\cap (U^-(\chi )\ot 1)\big)\fiee \]
and that
$L^\chi (\Lambda )$ is $\ndZ ^I$-graded.
For all $\al \in \ndZ ^I$ let
\[ I^\chi (\Lambda )_\al =M^\chi (\Lambda )_\al \cap
I^\chi (\Lambda ), \quad
L^\chi (\Lambda )_\al =M^\chi (\Lambda )_\al /I^\chi (\Lambda )_\al .\]
Since $M^\chi (\Lambda )_0=\fiee v_\Lambda $,
and any $\ndZ ^I$-graded quotient of
$M^\chi (\Lambda )$ by a $U(\chi )\ot _\fie \fiee $-submodule containing
$v_\Lambda $ is zero,
$L^\chi (\Lambda )$ is the unique simple $\ndZ ^I$-graded
$U(\chi )\ot _\fie \fiee $-module
quotient of $M^\chi (\Lambda )$.

\begin{defin}\label{de:fchar}
  Let $\Lambda \in \charUz $ and $V$ a $\ndZ ^I$-graded subquotient of
  $M^\chi (\Lambda )$. The \textit{(formal) character of}
  $V$ is the sum
  \[ \fch{V}=\sum _{\al \in \ndN _0^I} (\dim V_{-\al }) e^{-\al }, \]
  where $e$ is a formal variable.
\end{defin}

Eq.~\eqref{eq:MLgrading} implies that
\begin{align}
  \fch{M^\chi (\Lambda )}=\sum _{\al \in \ndN _0^I}\dim U^-(\chi )_{-\al }
  e^{-\al }
  \label{eq:chML}
\end{align}
for all $\Lambda \in \charUz $.

\begin{remar} \label{re:fch}
  For all $\al ,\beta \in \ndZ ^I$ we let $e^\al e^\beta =e^{\al +\beta }$. Thus
  we can consider formal characters as elements of the ring
  $\cup _{\al \in \ndN _0^I}e^\al \ndZ [[e^{-\al _i}\,|\,i\in I]]$, where
  $e^\al \ndZ [[e^{-\al _i}\,|\,i\in I]]\subset
  e^{\al +\beta }\ndZ [[e^{-\al _i}\,|\,i\in I]]$ for all $\al ,\beta \in \ndN _0^I$
  in the natural way.
\end{remar}

{\em From now on let $\chi \in \cX $, $p\in I$, and
$\bnd =\bfun \chi (\al _p)=\bfun {r_p(\chi )}(\al _p)$.
Assume that $\bnd <\infty $.}
Then $\chi $ and $r_p(\chi )$ are $p$-finite.
We deduce some phenomena which arise from the finiteness assumption on $\bnd $.

For all $\Lambda \in \charUz $ define $\VT _p^\chi (\Lambda )\in \charUz $ by
\begin{align}
  \VT _p^\chi (\Lambda )(K_\al L_\beta )=\Lambda (K_{\s _p^{r_p(\chi )}(\al )}
  L_{\s _p^{r_p(\chi )}(\beta )})
  \frac{r_p(\chi ) (\al ,\al _p)^{\bnd -1}}{r_p(\chi )(\al _p,\beta )^{\bnd -1}}
  \label{eq:tpLambda}
\end{align}
for all $\al ,\beta \in \ndZ ^I$. By Eq.~\eqref{eq:w*chi} this is equivalent
to
\begin{align}
  \VT _p^\chi (\Lambda )(K_{\s ^\chi _p(\al )}L_{\s _p^\chi (\beta )})
  =\Lambda (K_\al L_\beta )
  \frac
  {\chi (\al _p,\beta )^{\bnd -1}}
  {\chi (\al ,\al _p)^{\bnd -1}}
  \label{eq:tpLambda1}
\end{align}
for all $\al ,\beta \in \ndZ ^I$.

Recall the definition of $\rhomap \chi $ from Def.~\ref{de:rhomap}.

\begin{lemma}
  Let $\Lambda \in \charUz $. Then
  \begin{align*}
    \rhomap {r_p(\chi )}(\s _p^\chi (\al ))\, \VT _p^\chi (\Lambda )
    (K_{\s _p^\chi (\al )} L_{\s _p^\chi (\al )}^{-1})=
    \rhomap {\chi }(\al ) \, \Lambda (K_\al L_{\al }^{-1})
  \end{align*}
  for all $\al \in \ndZ ^I$ and $p\in I$.
  \label{le:VTinv}
\end{lemma}

\begin{proof}
  Insert Eq.~\eqref{eq:tpLambda1} and use Lemma~\ref{le:rho}.
\end{proof}

\begin{examp}
  Let $C$ be a symmetrizable Cartan matrix and $q\in \fienz $, $\chi
  \in \cX $ as in the second part of Ex.~\ref{ex:Cartan}.
  In particular, $\chi (\al _i,\al _j)=q^{d_i c_{i j}}$. Let $p\in I$.
  Assume that $\bnd =\bfun \chi (\al _p)<\infty $. Then $q^{2d_p \bnd
  }=1$, and hence $q^{2\bnd (\al ,\al _p)}=1$ for all $\al \in \ndZ
  ^I$. Further, $r_p(\chi )=\chi $.

  Let $\Lambda \in \charUz $. Assume that $\Lambda (K_\al L_\al
  ^{-1})=q^{2(\al ,\lambda )}$ for some $\lambda $ in the weight
  lattice. Then
  \begin{align*}
    \VT _p^\chi (\Lambda )(K_\al L_\al ^{-1})
    =&\Lambda (K_{\s _p^\chi (\al )} L_{\s _p^\chi (\al )}^{-1}) q^{2(b-1)(\al
    ,\al _p)}\\
    =&q^{2(\s _p^\chi (\al ),\lambda )}q^{-2(\al ,\al _p)}
    =q^{2(\al ,\s _p^\chi (\lambda )-\al _p)},
  \end{align*}
  which recovers the dot action of the Weyl group on the weight
  lattice.
\end{examp}

If we consider a composition
$\VT _i^{\chi '}\VT _j^{\chi ''}$,
where $i,j\in I$, $\chi ',\chi ''\in \cX $,
then we will always assume that
\[ \chi '=r_j(\chi ''). \]
For simplicity, we will
omit the upper index $\chi '$ if it is uniquely determined
by another bicharacter in
the expression.

\begin{lemma}
  Let $\Lambda \in \charUz $. Then $\VT _p \VT _p^\chi (\Lambda )
  =\Lambda $.
  \label{le:VTMrel1}
\end{lemma}

\begin{proof}
  By Eqs.~\eqref{eq:tpLambda} and \eqref{eq:tpLambda1},
  and since $r_p^2(\chi )=\chi $,
  \begin{align*}
    &\VT _p^{r_p(\chi )}\VT _p^\chi (\Lambda )(K_\al L_\beta )=
    \VT _p^\chi (\Lambda )(K_{\s _p^\chi (\al )}L_{\s _p^\chi (\beta )})
    \frac{\chi (\al ,\al _p)^{\bnd -1}}{\chi (\al _p,\beta )^{\bnd -1}}
    =\Lambda (K_\al L_\beta )
  \end{align*}
  for all $\al ,\beta \in \ndZ ^I$. This proves the lemma.
\end{proof}

\begin{lemma}
  \label{le:MLmap}
  Let $\Lambda \in \charUz $.
  There exist unique $\fiee $-linear maps
  \begin{gather*}
    \VTM _p=\VTM ^\chi _{p,\Lambda },\,
    \VTM ^-_p=\VTM ^{\chi ,-}_{p,\Lambda }: \,\,
    M^{r_p(\chi )}(\VT _p^\chi (\Lambda ))\to M^{\chi }(\Lambda ),\\
    \intertext{such that for all $u\in U(\chi )$,}
    \VTM _p(uv_{\VT _p^\chi (\Lambda )})=
    \LT _p(u)F_p^{\bnd -1}v_{\Lambda },\quad
    \VTM ^-_p(uv_{\VT _p^\chi (\Lambda )})=
    \LT ^-_p(u)F_p^{\bnd -1}v_{\Lambda }.
  \end{gather*}
  If $V\subset M^{r_p(\chi )}(\VT _p^\chi (\Lambda ))$ is a
  $U(r_p(\chi ))$-submodule, then $\VTM _p(V), \VTM ^-_p(V)$ are
  $U(\chi )$-submodules of $M^{\chi }(\Lambda )$.
\end{lemma}

\begin{proof}
  The uniqueness of the maps $\VTM _p$, $\VTM ^-_p$ is clear.
  We prove that $\VTM _p$ is well-defined. The proof for $\VTM ^-_p$ is
  analogous.

  Let $\chi '=r_p(\chi )$ and $\Lambda '=\VT _p^\chi (\Lambda )$.
  By Lemma~\ref{le:MLiso} and Thm.~\ref{th:PBWtau},
  \[ M^{\chi }(\Lambda )_{\al _j+a\al _p}=0=
  M^{\chi }(\Lambda )_{-\bnd \al _p} \quad 
  \text{for all $a\in \ndZ $.}
  \]
  Thus, since $\deg \LT _p(E_j)=\s _p^{\chi '}(\al _j)$ for $E_j\in U(\chi ')$,
  \begin{align*}
    \LT _p(E_j)F_p^{\bnd -1}v_\Lambda \in 
    M^{\chi }(\Lambda )_{\al _j+(1-\bnd -c^{\chi '}_{pj})\al _p}=0
    \quad \text{for all $j\in I$.}
  \end{align*}
  Moreover, Eqs.~\eqref{eq:KFrel}, \eqref{eq:tpLambda1} give that
  \begin{align*}
    K_{\al }L_{\beta }F_p^{\bnd -1} v_{\Lambda }
    =\chi (\al ,\al _p)^{1-\bnd }\, \chi (\al _p,\beta )^{\bnd -1}\,
    \Lambda (K_\al L_\beta ) F_p^{\bnd -1}v_{\Lambda }\qquad &\\
    =\Lambda '(K_{\s _p^\chi (\al )}
    L_{\s _p^\chi (\beta )}) F_p^{\bnd -1}v_\Lambda &
  \end{align*}
  for all $\al ,\beta \in \ndZ ^I$. Hence
  $\LT _p(u)F_p^{\bnd -1}v_{\Lambda }=\VTM _p(\Lambda '(u)v_{\Lambda '})$
  for all $u\in \Uz $. Therefore $\VTM _p$ is well-defined.

  The last claim of the lemma follows from the equations
  \begin{align}
    \VTM _p(uv)=\LT _p(u)\VTM _p(v),\quad 
    \VTM ^-_p(uv)=\LT ^-_p(u)\VTM ^-_p(v),
    \label{eq:Tuv}
  \end{align}
  where $u\in U(r_p(\chi ))$ and $v\in M^{r_p(\chi )}(\VT _p^\chi (\Lambda ))$,
  and from the surjectivity of the maps $\LT _p,\LT ^-_p$.
\end{proof}

\begin{lemma}
  \label{le:VTMrel2}
  Let $\Lambda \in \charUz $ and $j,k\in I$ such that $j\not=k$.
  Assume that
  $m_{j,k}^\chi =|(\ndN _0\al _j+\ndN _0\al _k)\cap R^\chi _+|<\infty $
  and that $\bfun \chi (\al )<\infty $ for all
  $\al \in (\ndN _0\al _j+\ndN _0\al _k)\cap R^\chi _+$.
  Then
  \begin{align}
    \label{eq:VTCox}
    (\VT _j\VT _k)^{m _{j,k}^\chi -1}\VT _j\VT _k^\chi (\Lambda )=\Lambda .
  \end{align}
\end{lemma}

\begin{proof}
  Let $m=m^\chi _{j,k}$, and let
  $i_0,i_1,\dots ,i_{m_{j,k}^\chi }\in \{j,k\}$
  such that $i_\nu =j$ if $\nu $ is even and $i_\nu =k$ if $\nu $ is odd.
  Let $\chi '=r_{i_0}r_{i_1}\cdots r_{i_{m-1}}(\chi )=
  r_{i_1}r_{i_2}\cdots r_{i_m}(\chi )$ (by (R4)),
  and
  \[ \Lambda '=\VT _{i_0}\VT _{i_1}\cdots \VT _{i_{m-1}}^\chi (\Lambda ),\quad
  \Lambda ''=\VT _{i_1}\VT _{i_2}\cdots \VT _{i_m}^\chi (\Lambda ). \]
  These definitions make sense, since $\bfun \chi (\al )<\infty $ for all
  $\al \in (\ndN _0\al _j+\ndN _0\al _k)\cap R^\chi _+$.
  Lemma~\ref{le:VTMrel1} implies that
  the claim of the lemma is equivalent to $\Lambda '=\Lambda ''$.

  Let $\VTM '=\VTM _{i_0}\VTM _{i_1}\cdots \VTM _{i_{m-1}}:
  M^{\chi '}(\Lambda ')\to M^\chi (\Lambda )$ and
  $\VTM ''=\VTM _{i_1}\VTM _{i_2}\cdots \VTM _{i_m}:
  M^{\chi '}(\Lambda '')\to M^\chi (\Lambda )$.
  For all $\nu \in \{1,2,\dots ,m\}$ let
  \[ \beta '_\nu =1_\chi \s _{i_0}\s _{i_1}\cdots \s _{i_{\nu -2}}
  (\al _{i_{\nu -1}}), \quad
  \beta ''_\nu =1_\chi \s _{i_1}\s _{i_2}\cdots \s _{i_{\nu -1}}
  (\al _{i_\nu }). \]
  By definition of $\VTM _j$ and $\VTM _k$,
  \begin{align*}
    \VTM ''(v_{\Lambda ''})=&\VTM _{i_1}\cdots \VTM _{i_{m-1}}
    (F_{i_m}^{\bfun {\chi '}(\al _{i_m})-1}v_{\VT _{i_m}^{\chi '}(\Lambda '')})
    =\cdots
    \\
    =&F_{\beta ''_m}^{\bfun \chi (\beta ''_m)-1}\cdots
    F_{\beta ''_2}^{\bfun \chi (\beta ''_2)-1}
    F_{\beta ''_1}^{\bfun \chi (\beta ''_1)-1}v_\Lambda ,\\
    \VTM '(v_{\Lambda '})=&F_{\beta '_m}^{\bfun \chi (\beta '_m)-1}\cdots
    F_{\beta '_2}^{\bfun \chi (\beta '_2)-1}
    F_{\beta '_1}^{\bfun \chi (\beta '_1)-1}v_\Lambda .
  \end{align*}
  Both expressions are nonzero by Thm.~\ref{th:PBWtau}.
  Since 
  \[ \{\beta '_\nu \,|\,1\le \nu \le m\}=\{\beta ''_\nu \,|\,1\le \nu \le m\}
  =R^\chi _+\cap (\ndN _0\al _j+\ndN _0\al _k),\]
  we obtain that
  \begin{itemize}
    \item [($*$)]
    $\VTM '(\fiee v_{\Lambda '})$ and $\VTM ''(\fiee v_{\Lambda ''})$ are
    isomorphic $\Uz $-modules.
  \end{itemize}
  By Thm.~\ref{th:Coxgr},
  \[ \LT _{i_0}\LT _{i_1}\cdots \LT _{i_{m-1}}(u_0)=
  \LT _{i_1}\LT _{i_2}\cdots \LT _{i_m}(u_0)\quad
  \text{for all $u_0\in \Uz $.} \]
  Hence
  Lemma~\ref{le:MLmap} yields that
  \begin{align*}
    \Lambda '(u_0)\VTM '(v_{\Lambda '})=\VTM '(u_0v_{\Lambda '})
    =&\LT _{i_0}\LT _{i_1}\cdots \LT _{i_{m-1}}(u_0)
    \VTM '(v_{\Lambda '}),\\
    \Lambda ''(u_0)\VTM ''(v_{\Lambda ''})=\VTM ''(u_0v_{\Lambda ''})
    =&\LT _{i_0}\LT _{i_1}\cdots \LT _{i_{m-1}}(u_0)
    \VTM ''(v_{\Lambda ''}).
  \end{align*}
  Thus $\Lambda '=\Lambda ''$ by ($*$).
  This proves the lemma.
\end{proof}

\begin{remar}
  In view of Thm.~\ref{th:Coxgr} and Lemmata~\ref{le:VTMrel1},
  \ref{le:VTMrel2} we can say that Eq.~\eqref{eq:tpLambda} defines an action
  of the groupoid $\Wg (\chi )$ on $\charUz $. Then Lemma~\ref{le:VTinv}
  says that the numbers $\rhomap \chi (\al )\Lambda (K_\al L_\al ^{-1})$,
  where $\al \in \ndZ ^I$, are invariants of this action.
\end{remar}

In general, the maps $\VTM _p$ and $\VTM ^-_p$ are not isomorphisms.

\begin{propo}
  Assume that
  $\Lambda (K_pL_p^{-1})\not=\chi (\al _p,\al _p)^{t-1}$ for all
  $t\in \{1,2,\dots ,\bnd -1\}$. Then
  $\VTM _p,\VTM ^-_p:
  \,M^{r_p(\chi )}(\VT _p^\chi (\Lambda ))\to M^\chi (\Lambda )$
  are isomorphisms of vector spaces over $\fiee $.
  \label{pr:VTMiso}
\end{propo}

\begin{proof}
  Let $q=\chi (\al _p,\al _p)$,
  $\chi '=r_p(\chi )$, $\Lambda '=\VT _p^\chi (\Lambda )$, and
  $\VTM '=\VTM ^\chi _{p,\Lambda } \VTM ^{\chi ',-}_{p,\Lambda '}$.
  By Lemma~\ref{le:VTMrel1} and since $r_p^2(\chi )=\chi $, $\VTM '$
  is a $U(\chi )$-module endomorphism of $M^\chi (\Lambda )$.
  We calculate $\VTM '(v_\Lambda )$.
  \begin{align*}
    \VTM '(v_\Lambda )
    =&\VTM _p(F_p^{\bnd -1}v_{\Lambda '})
    =\LT _p(F_p^{\bnd -1})F_p^{\bnd -1}v_\Lambda \\
    =&(K_p^{-1}E_p)^{\bnd -1}F_p^{\bnd -1}v_\Lambda
    =q^{(\bnd -2)(\bnd -1)/2}\Lambda (K_p^{1-\bnd })E_p^{\bnd -1}F_p^{\bnd -1}v_\Lambda \\
    =&q^{(\bnd -2)(\bnd -1)/2}\Lambda (K_p^{1-\bnd })\qfact{\bnd -1}{q}
    \prod _{t=1}^{\bnd -1}(q^{t+1-\bnd }\Lambda (K_p)-\Lambda (L_p))v_\Lambda
  \end{align*}
  by Lemma~\ref{le:EmFn}. By assumption, $\VTM '(v_\Lambda )\not=0$, and hence
  $\VTM '$ is a nonzero multiple of $\id _{M^\chi (\Lambda )}$. Therefore
  $\VTM _p$ is an isomorphism. The proof for $\VTM _p^-$ is analogous.
\end{proof}

\begin{lemma}
  \label{le:hwvector}
  Let $t\in \{1,2,\dots ,\bnd -1\}$. Let $q=\chi (\al _p,\al _p)$.
  Assume that $\Lambda (K_pL_p^{-1})=q^{t-1}$.
  Then in $M^\chi (\Lambda )$
  \begin{align}
    E_pF_p^m v_\Lambda =\qnum{m}{q}\Lambda (L_p)(q^{t-m}-1)F_p^{m-1} v_\Lambda \quad
    \text{for all $m\in \ndN _0$.}
    \label{eq:EpFpmv}
  \end{align}
  In particular, if $q\not=1$, then
  $EF_p^mv_\Lambda =\coun (E)F_p^mv_\Lambda $ for all $E\in U^+(\chi )$ if and
  only if $m=0$, $m=t$ or $m\ge \bnd $.
  If $q=1$, then $EF_p^mv_\Lambda =\coun (E)F_p^mv_\Lambda $ for all $E\in
  U^+(\chi )$, $m\in \ndN _0$.
\end{lemma}

\begin{proof}
  Eq.~\eqref{eq:EpFpmv} follows from Lemma~\ref{le:EmFn}.
  By definition of $\bnd =\bfun \chi (\al _p)$, either $q\not=1$ and $q$ is a
  primitive $\bnd $-th root of $1$, or $q=1$ and $\bnd =\mathrm{char}\,\fie $.
  Therefore, if $q\not=1$ and $m\in \{0,1,\dots ,\bnd -1\}$, then $q^{t-m}=1$ if
  and only if $t=m$.
  If $E=E_i$ with $i\not=p$, then $EF_p^m=0$ by Eqs.~\eqref{eq:MLgrading},
  \eqref{eq:MLgrading2}.
  The rest is a
  consequence of Lemma~\ref{le:Eheight}(i).
\end{proof}

\begin{propo}
  \label{pr:VTMker}
  Assume that
  $\Lambda (K_pL_p^{-1})=\chi (\al _p,\al _p)^{t-1}$ for some
  $t\in \{1,2,\dots ,\bnd -1\}$.
  
  (i) If $\chi (\al _p,\al _p)\not=1$, then $t$ is unique, and
  \begin{align*}
    \ker \VTM ^{\chi }_{p,\Lambda }=\ker \VTM ^{\chi ,-}_{p,\Lambda }
    = & \, U^-(\chi )F_p^{\bnd -t}\ot \fiee _{\VT _p^\chi (\Lambda )},\\
    \im \VTM ^{\chi }_{p,\Lambda }=\im \VTM ^{\chi ,-}_{p,\Lambda }
    = & \, U^-(\chi )F_p^t \ot \fiee _\Lambda .
  \end{align*}
  
  (ii) If $\chi (\al _p,\al _p)=1$, then $\mathrm{char}\,\fie =\bnd >0$ and
  \begin{align*}
    \ker \VTM ^{\chi }_{p,\Lambda }=\ker \VTM ^{\chi ,-}_{p,\Lambda }
    = & \, U^-(\chi )F_p\ot \fiee _{\VT _p^\chi (\Lambda )},\\
    \im \VTM ^{\chi }_{p,\Lambda }=\im \VTM ^{\chi ,-}_{p,\Lambda }
    = & \, U^-(\chi )F_p^{\bnd -1} \ot \fiee _\Lambda .
  \end{align*}
\end{propo}

\begin{proof}
  Let $q=\chi (\al _p,\al _p)$, $\chi '=r_p(\chi )$,
  and $\Lambda '=\VT _p^\chi (\Lambda )$.
  We prove the claims about $\VTM ^\chi _{p,\Lambda }$ in (i).
  The rest is analogous.
  By Lemma~\ref{le:EmFn}, for all
  $m\in \{0,1,\dots ,\bnd -1\}$,
  \begin{align*}
    \VTM ^\chi _{p,\Lambda }(F_p^mv_{\Lambda '})
    =&\LT _p(F_p^m)F_p^{\bnd -1}v_\Lambda
    =(K_p^{-1}E_p)^mF_p^{\bnd -1}v_\Lambda \\
    =&aE_p^mF_p^{\bnd -1}v_\Lambda =a'F_p^{\bnd -1-m}
    \prod _{j=1}^m (q^{j+1-\bnd }\Lambda (K_pL_p^{-1})-1) v_\Lambda
  \end{align*}
  for some $a,a'\in \fienz $.
  Thus, $\VTM ^\chi _{p,\Lambda }(F_p^mv_{\Lambda '})=0$ if and only if
  $q^j=q^{\bnd -1}\Lambda (K_pL_p^{-1})^{-1}$ for some $j\in \{1,2,\dots ,m\}$.
  By the assumption on $\Lambda (K_pL_p^{-1})$, this is equivalent to
  $j=\bnd -t$ for some $j\in \{1,2,\dots ,m\}$. Therefore
  $\VTM ^\chi _{p,\Lambda }(F_p^mv_{\Lambda '})=0$ if and only if
  $m\ge \bnd -t$.

  Let $F\in U^-(\chi ')$.
  By Lemma~\ref{le:Eheight}(ii), there exist unique
  $F'_m\in U^-_{p,K}(\chi ')$, where $m\in \{0,1,\dots ,\bnd -1\}$, such that
  $F=\sum _{m=0}^{\bnd -1}F'_mF_p^m$.
  By the previous paragraph, and since $\Lambda (K_pL_p^{-1})\in \fie $,
  for each $m\in \{0,1,\dots ,\bnd -1-t\}$ there is
  a unique $a_m\in \fienz $ such that
  \[ \VTM ^\chi _{p,\Lambda }(F v_{\Lambda '})=\sum _{m=0}^{\bnd -1-t}
  a_m\LT _p (F'_m)F_p^{\bnd -1-m}v_\Lambda \in U^-(\chi )\ot 1_\Lambda .  \]
  By Lemma~\ref{le:Eheight}(ii), the latter
  expression is zero if and only if $\LT _p(F'_m)=0$
  for all $m\in \{0,1,\dots ,\bnd -1-t\}$.
  Therefore, Lemma~\ref{le:TpU+U+} and relations $F'_m\in U^-_{i,K}(\chi ')$
  imply that $\VTM ^\chi _{p,\Lambda }(F v_{\Lambda '})=0$ if and only if
  $F'_m=0$ for all $m\in \{0,1,\dots ,\bnd -1-t\}$.
  Hence
  $\ker \VTM ^{\chi }_{p,\Lambda }=U^-(\chi )F_p^{\bnd -t}\ot \fiee _{\Lambda '}$
  and $\im \VTM ^{\chi }_{p,\Lambda } = U^-(\chi )F_p^t \ot \fiee _\Lambda $.
\end{proof}

For all $w\in \Aut (\ndZ ^I)$ and $\al \in \ndZ ^I$ let
$w(e^\al )=e^{w(\al )}$, and extend this definition linearly on formal
characters. We investigate the effect of the maps $\VTM _p,\VTM ^-_p$
on formal characters.
For all $\chi '\in \cG (\chi )$ and $i\in I$ with
$\bfun{\chi '}(\al _i)<\infty $ let
$\sdot ^{\chi '}_i:\ndZ ^I\to \ndZ ^I$ be the affine transformation
\begin{align}
  \sdot _i^{\chi '}(\al )=\s _i^{\chi '}(\al )
  +(1-\bfun{\chi '}(\al _i))\al _i.
  \label{eq:sdot}
\end{align}
Note that then $\sdot _i^{r_i(\chi ')}\sdot _i^{\chi '}(\al )=\al $ for all
$\al \in \ndZ ^I$.

\begin{lemma}\label{le:Tpfch}
  Let $\Lambda \in \charUz $ and $\al \in \ndZ ^I$. Then
  \begin{align}
    \label{eq:Tpweight+}
    \VTM _p(M^{r_p(\chi )}(\VT ^\chi _p(\Lambda ))_\al )
    \subset &M^\chi (\Lambda )_{\sdot _p^{r_p(\chi )}(\al )},\\
    \label{eq:Tpweight-}
    \VTM ^-_p(M^{r_p(\chi )}(\VT ^\chi _p(\Lambda ))_\al )
    \subset &M^\chi (\Lambda )_{\sdot _p^{r_p(\chi )}(\al )}.
  \end{align}
  In particular,
  \begin{align}
    \fch{M^\chi (\Lambda )} =
    \sdot _p^{r_p(\chi )}
    (\fch{M^{r_p(\chi )}(\VT _p^\chi (\Lambda ))}).
    \label{eq:fchM}
  \end{align}
\end{lemma}

\begin{proof}
  Let $\Lambda '=\VT ^\chi _p(\Lambda )$ and $u\in U(r_p(\chi ))_\al $. Then
  \[ \VTM _p(uv_{\Lambda '})=\LT _p(u)F_p^{\bnd -1}v_\Lambda
  \in U(\chi )_{\s _p^{r_p(\chi )}(\al )}U(\chi )_{(1-\bnd )\al _p}v_\Lambda \]
  by Prop.~\ref{pr:LTdeg}.
  This proves Eq.~\eqref{eq:Tpweight+}, since $v_\Lambda \in M^\chi (\Lambda )_0$,
  The proof of Eq.~\eqref{eq:Tpweight-} is similar. By Eq.~\eqref{eq:chML},
  $\fch M^\chi (\Lambda )$ does not depend on $\Lambda $. Hence by
  Prop.~\ref{pr:VTMiso} we may assume that $\VTM _p$ is an isomorphism.
  Then Eq.~\eqref{eq:fchM} follows from Eq.~\eqref{eq:Tpweight+}.
\end{proof}

\begin{lemma} \label{le:chsub}
  Let $\Lambda \in \charUz $ and $t\in \{1,2,\dots ,\bnd -1\}$.
  Assume that
  $\Lambda (K_pL_p^{-1})=\chi (\al _p,\al _p)^{t-1}$.
  Then $V=U^-(\chi )F_p^t \ot \fiee _\Lambda $
  is a $U(\chi )\ot \fiee $-submodule
  of $M^\chi (\Lambda )$ with
  \[ \fch{V}=\fch{M^\chi }(\Lambda )
  \frac{e^{-t\al _p}-e^{-\bnd \al _p}}{1-e^{-\bnd \al _p}}. \]
\end{lemma}

\begin{proof}
  The formal character of the subspace
  $\oplus _{m=t}^{\bnd -1}F_p^m$ of $M^\chi (\Lambda )$ is
  \[ e^{-t\al _p}+e^{-(t+1)\al _p}+\cdots +e^{-(\bnd -1)\al _p}=
  \frac{e^{-t\al _p}-e^{-\bnd \al _p}}{1-e^{-\al _p}}. \]
  Thus the lemma is a consequence of Lemmata~\ref{le:hwvector},
  \ref{le:MLiso} and Eqs.~\eqref{eq:PBWbasis}, \eqref{eq:roots}.
\end{proof}

\textit{In the rest of this section
assume that $\chi \in \cX _4$.}
Let $n=|R^\chi _+|$ and
$i_1,\dots ,i_n\in I$ with $\ell (1_\chi \s _{i_1}\cdots \s_{i_n})=n$.
Recall the definitions of $\beta _\nu $ and $F_{\beta _\nu }$, where
$1\le \nu \le n$,
from Eq.~\eqref{eq:Ebetak}.
Let $\Lambda \in \charUz $.
We characterize irreducible Verma modules
(see also Lemma~\ref{le:subfch}).

\begin{propo}
  \label{pr:M=L}
  Assume that
  \begin{align}
    \label{eq:MLass}
    \prod _{\nu =1}^n \prod _{t=1}^{\bfun \chi (\beta _\nu )-1}
    \big( \rhomap \chi (\beta _\nu )\Lambda (K_{\beta _\nu }
    L_{\beta _\nu }^{-1})-\chi (\beta _\nu ,\beta _\nu )^t\big)\not=0.
  \end{align}
  Then $I^\chi (\Lambda )=0$.
\end{propo}

\begin{proof}
  For all $\nu \in \{1,2,\dots ,n\}$ let
  $\chi _\nu =r_{i_{\nu -1}}\cdots r_{i_2}r_{i_1}(\chi )$ and
  $\Lambda _\nu =\VT _{i_{\nu -1}}\cdots \VT _{i_2}
  \VT _{i_1}^\chi (\Lambda )$.
  By Lemma~\ref{le:VTinv} and Eq.~\eqref{eq:hghtrpchi},
  Eq.~\eqref{eq:MLass} is equivalent to
  \[ \rhomap {\chi _\nu } (\al _{i_\nu })\Lambda _\nu (K_{i_\nu }L_{i_\nu }^{-1})\not=
  \chi _\nu (\al _{i_\nu },\al _{i_\nu })^t \]
  for all $\nu \in \{1,2,\dots ,n\}$,
  $t\in \{1,2,\dots ,\bfun {\chi _\nu }(\al _{i_\nu })-1\}$.
  Hence, by Prop.~\ref{pr:VTMiso}, the map
  \[ \VTM _{i_1}\VTM _{i_2}\cdots \VTM _{i_n}: M^{r_{i_n}(\chi _n)}
  (\VT _{i_n}(\Lambda _n))\to M^\chi (\Lambda ) \]
  is an isomorphism. Thus
  $v=F_{\beta _n}^{\bfun \chi (\beta _n)-1}
  \cdots F_{\beta _2}^{\bfun \chi (\beta _2)-1}
  F_{\beta _1}^{\bfun \chi (\beta _1)-1} v_\Lambda \not=0$
  and $(U^+(\chi )\ot _\fie \fiee )v=M^\chi (\Lambda )$.
  Since $v$ is contained in any nonzero
  $U(\chi )\ot _\fie \fiee $-submodule of $M^\chi (\Lambda )$ by
  Thms.~\ref{th:EErel}, \ref{th:PBWtau}, it follows that $I^\chi (\Lambda )=0$.
\end{proof}

\section{The Shapovalov form}
\label{sec:shapdet}

We discuss the analog of the Shapovalov form for
the algebras $U(\chi )$ following the construction in
\cite[3.4.10]{b-Joseph}.

Let $\chi \in \cX $. By Prop.~\ref{pr:tridec}, there exists a decomposition
\[ \cU (\chi )=\Big(\sum _{i\in I}F_i\cU (\chi )+\sum _{i\in I}\cU (\chi )E_i
\Big)
\oplus \Uz  \]
and hence a unique projection
\[ \HCmap \chi :\cU (\chi )\to \Uz  \]
with kernel $\sum _{i\in I}F_i\cU (\chi )+\sum _{i\in I}\cU (\chi )E_i$.
This map is commonly known as
the \textit{Harish-Chandra map}. By definition, $\HCmap \chi $
satisfies the property
\begin{align}
  \HCmap \chi (u_-uu_+)=\coun (u_-)\HCmap \chi (u)\coun (u_+)
  \label{eq:HCprop}
\end{align}
for all $u_-\in U^-(\chi )$, $u\in U(\chi )$, $u_+\in U^+(\chi )$.
Since $\aaaU (u)=u$ for all $u\in \Uz $,
\begin{align}
  \HCmap \chi (\aaaU (u))=\HCmap \chi (u)\qquad
  \text{for all $u\in \cU (\chi )$.}
  \label{eq:HCprop2}
\end{align}
The bilinear map
\begin{align}
  \Shf : \cU (\chi )\times \cU (\chi )\to \Uz ,\qquad
  \Shf (u,v)=\HCmap \chi (\aaaU (u)v),
  \label{eq:Shf}
\end{align}
is called the \textit{Shapovalov form}.
By Eq.~\eqref{eq:HCprop2} and since $\aaaU ^2=\id $,
\begin{align}
  \Shf (u,v)=\Shf (v,u) \quad \text{for all $u,v\in \cU (\chi )$.}
  \label{eq:Shfprop}
\end{align}
Moreover, by definition of $\Shf $ and $\HCmap \chi $,
\begin{align}
  \Shf (u,v)=0 \quad  \text{if $u\in \sum _{i\in I}\cU (\chi )E_i$ or
  $v\in \sum _{i\in I}\cU (\chi )E_i$.}
  \label{eq:Shfprop2}
\end{align}
Recall the definitions of $U(\chi )$, $\cI ^+(\chi )$ and $\cI ^-(\chi
)$ from Sect.~\ref{sec:DD}.
Since $\cU (\chi )\cI ^+(\chi )\cU (\chi )+
\cU (\chi )\cI ^-(\chi )\cU (\chi )\subset \ker \HCmap \chi $,
$\HCmap \chi $ and $\Shf $ induce maps
\[ \HCmap \chi : U(\chi )\to \Uz ,\qquad
\Shf : U(\chi )\times U(\chi )\to \Uz . \]
The map $\HCmap \chi $ is $\ndZ ^I$-homogeneous, that is,
$\HCmap \chi (u)=0$
for all $u\in U(\chi )_\al $, where $\al \in \ndZ ^I\setminus \{0\}$.
The map $\aaaU $ reverses degrees, that is,
$\aaaU (u)\in U(\chi)_{-\al }$ for all $u\in U(\chi )_\al $, where
$\al \in \ndZ ^I$. Therefore, for all $\al ,\beta \in \ndZ ^I$, where $\al
\not=\beta $, we get
\begin{align}
  \Shf (u,v)=0 \quad \text{for all $u\in U(\chi )_\al $,
  $v\in U(\chi )_\beta $ \quad $(\al \not=\beta )$.}
  \label{eq:Shfprop3}
\end{align}

\begin{defin}\label{de:Shapdet}
  The family of determinants
  \[ \det \nolimits ^\chi _\al =
  \det \Shf (F'_i,F'_j)_{i,j\in \{1,\dots ,k\}}\in \Uz /\fienz ,\]
  where $\al \in \ndN _0^I$, $k=\dim U^-(\chi )_{-\al }$, and
  $\{F'_1,F'_2,\dots ,F'_k\}$ is a basis of
  $U^-(\chi )_{-\al }$,
  is called the \textit{Shapovalov determinant} of $U(\chi )$.
\end{defin}

\begin{remar}\label{re:Shdet}
  Let $\al \in \ndN _0^I$ and $k=\dim U^-(\chi )_{-\al }$.
  By the above considerations,
  $\Shf :U^-(\chi )_{-\al }\times U^-(\chi )_{-\al }\to \Uz $
  is a symmetric bilinear form for all $\al \in \ndN _0^I$.
  Let $F'=\{F'_1,F'_2,\dots ,F'_k\}$ be a basis of
  $U^-(\chi )_{-\al }$,
  and let $d(F')=\det \Shf (F'_i,F'_j)_{i,j\in \{1,\dots ,k\}} \in \Uz $.
  Then $d(A'F')=(\det A')^2d(F')$ for all $A'\in \mathrm{GL}(k,\fie )$, and
  hence $\det ^\chi _\al =d(F')/\fienz $
  does not depend on the choice of the basis $F'$ of
  $U^-(\chi )_{-\al }$.
\end{remar}

\begin{lemma}
  \label{le:Uzideal}
  Let $\chi \in \cX $. Let $J$ be an ideal of $\Uz $. Assume that
  $J$ is contained in
  the center of $U(\chi )$. Let $J_U$
  be the ideal of $U(\chi )$ generated by $J$.  Then $\Shf
  :U(\chi )\times U(\chi )\to \Uz $ induces a map
  $\Shf :U(\chi )/J_U\times U(\chi )/J_U\to \Uz /J$.
\end{lemma}

\begin{proof}
  Since $J$ is contained in the center of $U(\chi )$, triangular
  decomposition of $U(\chi )$ yields that
  $J_U=U^-(\chi )JU^+(\chi )$. This and $\aaaU (J)=J$
  imply the claim of the lemma.
\end{proof}

\begin{lemma}
  Let $\chi \in \cX $, $\al \in \ndZ ^I$,
  $E\in U^+(\chi )_\al $, and $F\in U^-(\chi )_{-\al }$.
  Then
  \[ \Shf (\Omega (E),F)\in
  \sum _{\beta ,\gamma \in \ndN _0^I,\, \beta +\gamma =\al }
  \fie K_\beta L_\gamma \]
  Moreover, the coefficients of $K_\al $ and $L_\al $ of $\Shf (\aaaU (E),F)$
  are $\sHp (E,S(F))$ and $\sHp (E,F)$, respectively.
  \label{le:Shfcoeffs}
\end{lemma}

\begin{proof}
  If $\al \notin \ndN _0^I$ or $\al =0$, then $U ^+(\chi )_\al =0$ or
  $U ^+(\chi )_\al =\fie $. In this case the claim of the lemma holds by
  definition of $\Shf $. Assume now that $\al \in \ndN _0^I\setminus \{0\}$.
  Using Eqs.~\eqref{eq:KLrel}--\eqref{eq:EFrel},
  by induction on $\beta $ and $\beta '$ one can show that
  \[
  E'F'\in \sum _{\gamma _1,\gamma _2,\gamma _3,\gamma _4\in \ndN _0^I,
  \gamma _4-\gamma _1=\beta -\beta ',\gamma _2+\gamma _3+\gamma _4=\beta }
  U ^-(\chi )_{-\gamma _1}K_{\gamma _2}L_{\gamma _3}U ^+(\chi )_{\gamma
  _4}\]
  for all $\beta ,\beta '\in \ndN _0^I$ and $E'\in U ^+(\chi )_\beta $,
  $F'\in U ^-(\chi )_{-\beta '}$.
  This implies the first claim of the lemma by letting
  $\beta =\beta '=\al $ and $E'=E$, $F'=F$.
  The second claim follows from
  \begin{align*}
    \copr (E)-K_\al \ot E-E\ot 1\in &\mathop{\oplus }
    _{\beta ,\gamma \in \ndN _0^I, \beta +\gamma =\al ,\,\beta ,\gamma \not=0}
    U ^+(\chi )_\beta K_\gamma \ot U ^+(\chi )_\gamma ,\\
    \copr (F)-1 \ot F-F\ot L_\al \in &\mathop{\oplus }
    _{\beta ,\gamma \in \ndN _0^I, \beta +\gamma =\al ,\beta ,\gamma \not=0}
    U ^-(\chi )_{-\beta }\ot U ^-(\chi )_{-\gamma }L_\beta ,
  \end{align*}
  and from Eq.~\eqref{eq:Ucomm2} (with $x=E$, $y=F$)
  and Prop.~\ref{pr:sHpdef}(iii).
\end{proof}

Let $\fiee $ be a field extension of $\fie $.
The importance of the Shapovalov form arises from the fact that it induces a
form on the Verma modules $M^\chi (\Lambda )$ and on their simple quotients
$L^\chi (\Lambda )$, where $\Lambda \in \charUz $.

Let $\Lambda \in \charUz $. Define
\begin{align}
  \Lambda \Shf :U(\chi )\times U(\chi )\to \fiee ,\quad
  (u,v)\mapsto \Lambda (\Shf (u,v)).
  \label{eq:LSh}
\end{align}
By Eq.~\eqref{eq:HCprop},
\begin{align*}
  \Lambda \Shf (u_-u_0u_+,v_-v_0v_+)=&
  \coun (u_+)\coun (v_+)\Lambda (u_0\Shf (u_-,v_-)v_0)\\
  =&\Lambda (u_0)\Lambda (v_0)\coun (u_+)\coun (v_+)
  \Lambda \Shf (u_-,v_-).
\end{align*}
Thus, by Eq.~\eqref{eq:fieLambda},
$\Lambda \Shf $ induces a $\fiee $-bilinear form on
$M^\chi (\Lambda )$ by letting
\[ \Lambda \Shf :M(\Lambda )\times M(\Lambda )\to \fiee ,\quad
(u\ot 1_\Lambda , v\ot 1_\Lambda )\mapsto \Lambda \Shf (u,v) \]
for all $u,v\in U(\chi )$.
Moreover, Eq.~\eqref{eq:Shfprop3} gives that
\begin{align}\label{eq:LShf}
  \Lambda \Shf ( u\ot 1_\Lambda , v)
  =\Lambda \Shf (1\ot 1_\Lambda ,\aaaU (u)v) =0
\end{align}
for all $u\in U(\chi )$ and $v\in I^\chi (\Lambda )$,
since $\aaaU (u)v \in I^\chi (\Lambda )\subset
\oplus _{\al \not=0}M^\chi (\Lambda )_\al $.
Thus by Eq.~\eqref{eq:Shfprop}, $\Lambda \Shf $
induces a symmetric bilinear form on $L^\chi (\Lambda )$,
also denoted by $\Lambda \Shf $.
The radical of this form is a $\ndZ ^I$-graded
$U(\chi )\ot _\fie \fiee $-submodule of $L^\chi (\Lambda )$,
but does not contain $1\ot 1_\Lambda $,
and hence it is zero. Thus $\Lambda \Shf $ is a nondegenerate symmetric
bilinear form on $L^\chi (\Lambda )$.



{\em For the rest of this section let $\chi \in \cX _4$, $n=|R^\chi _+|$,
and $i_1,\dots ,i_n\in I$ with $\ell (1_\chi \s _{i_1}\cdots \s_{i_n})=n$.}
For all $\nu \in \{1,2,\dots ,n\}$ let
\[ \beta _\nu =1_\chi \s _{i_1}\s _{i_2}\cdots \s _{i_{\nu -1}}(\al _{i_\nu }),\qquad
\chi _\nu =r_{i_{\nu -1}}\cdots r_{i_2}r_{i_1}(\chi ). \]
For all $\nu \in \{1,2,\dots ,n\}$, $\al \in \ndN _0^I$,
and $t\in \{1,2,\dots ,\bfun \chi (\beta _\nu )-1\}$ let
\begin{equation}
\begin{aligned}
  \PF ^\chi (\al ,\beta _\nu ;t)=\Big|\Big\{(m_1,\dots ,m_n)\in \ndN _0^n\,\big|\,
  \sum _{\mu =1}^n m_\mu \beta _\mu =\al ,\,m_\nu \ge t,\quad &\\
  m_\mu <\bfun \chi (\beta _\mu )\quad \text{for all $\mu \in \{1,2,\dots ,n\}$}
  \Big\}\Big|.&
  \label{eq:PF}
\end{aligned}
\end{equation}

We will use two important facts on the function $\PF ^\chi $.

\begin{lemma} \label{le:P1}
  For all $\nu \in \{1,2,\dots ,n\}$ and
  $t\in \{1,2,\dots ,\bfun \chi (\beta _\nu)-1\}$,
  \[ \sum _{\al \in \ndN _0^I}\PF ^\chi (\al ,\beta _\nu ;t)e^{-\al }=
  \frac {e^{-t\beta _\nu}-e^{-\bfun \chi (\beta _\nu )\beta _\nu }}
  {1-e^{-\beta _\nu }}
  \prod _{\mu \in \{1,\dots ,n\},\, \mu \not=\nu }
  \frac {1-e^{-\bfun \chi (\beta _\mu )\beta _\mu }}
  {1-e^{-\beta _\mu }}. \]
\end{lemma}

\begin{proof}
  By Thm.~\ref{th:PBWtau},
  the two sides of the equation are two different expressions for the formal
  character of the subspace of $U^-(\chi )\ot \fiee $ spanned by the elements
  \[ \prod _{ {m_1,\dots ,m_n \atop m_\nu \ge t,\, 0\le m_\mu <\bfun \chi (\beta _\mu )
  \,\text{for all $\mu $}}}
  F_{\beta _1}^{m_1}F_{\beta _2}^{m_2}\cdots F_{\beta _n}^{m_n}. \]
\end{proof}

\begin{lemma}
  \label{le:P2}
  For all $\al \in \ndN _0^I$,
  \begin{align*}
    \al \dim U^-(\chi )_{-\al }=
    \sum _{\nu =1}^n \sum _{t=1}^{\bfun \chi (\beta _\nu )-1}
    \PF ^\chi (\al ,\beta _\nu ;t) \beta _\nu .
  \end{align*}
\end{lemma}

\begin{proof}
  By Thm.~\ref{th:PBWtau}, for each $\al \in \ndN _0^I$ there is a basis of
  $U^-(\chi )_{-\al }$ parametrized by the set
  \begin{equation}
    \Big\{(m_1,\dots ,m_n)\in \ndN _0^n\,\big|\,
    \sum _{\mu =1}^n m_\mu \beta _\mu =\al ,\,\,
    m_\mu <\bfun \chi (\beta _\mu )\,\, \text{for all $\mu $}\Big\}.
    \label{eq:dimU-}
  \end{equation}
  Each $(m_1,\dots ,m_n)$ in this set contributes to $\PF ^\chi (\al ,\beta _\nu ;t)$
  with a summand $1$, for all $\nu \in \{1,2,\dots ,n\}$ and
  $t\in \{1,2,\dots ,m_\nu \}$. Thus the claim of the
  lemma follows from the decomposition of the
  $\PF ^\chi (\al ,\beta _\nu ;t)$ into $1+1+\cdots +1$ by reordering the summands.
\end{proof}

\begin{lemma}
  \label{le:subfch}
  Let $\nu \in \{1,2,\dots ,n\}$,
  $t\in \{1,2,\dots ,\bfun \chi (\beta _\nu)-1\}$,
  and $\Lambda \in \charUz $.
  Assume that
  $\rhomap \chi (\beta _\nu )\Lambda (K_{\beta _\nu }
  L_{\beta _\nu }^{-1})
  =\chi (\beta _\nu ,\beta _\nu )^t$ and
  \[ \prod _{\mu =1}^{\nu -1} \prod _{m=1}^{\bfun \chi (\beta _\mu )-1}
  (\rhomap \chi (\beta _\mu )
  \Lambda (K_{\beta _\mu }L_{\beta _\mu }^{-1})-\chi (\beta _\mu ,
  \beta _\mu )^m)\not=0. \]
  Then $M^\chi (\Lambda )$ contains a $U(\chi )\ot \fiee $-submodule $V$
  with
  \begin{align}
    \fch V=\sum _{\al \in \ndN _0^I}\PF ^\chi (\al ,\beta _\nu ;t)e^{-\al }.
    \label{eq:fchV}
  \end{align}
  In particular, $0\not=V\subset I^\chi (\Lambda )$.
\end{lemma}

\begin{proof}
  We proceed by induction on $\nu $. Let first $\nu =1$.
  By Lemma~\ref{le:chsub},
  $V=U^-(\chi )F_{i_1}^t\ot \fiee _\Lambda $ is a
  $U(\chi )\ot \fiee $-submodule of $M^\chi (\Lambda )$. Then
  Eq.~\eqref{eq:fchV} follows from Thm.~\ref{th:PBWtau}.

  Assume now that $\nu \in \{2,3,\dots ,n\}$ and that the lemma holds for
  $\nu -1$.
  Let $\chi _\mu =r_{i_{\mu -1}}\cdots
  r_{i_2}r_{i_1}(\chi )$ and
  $\Lambda _\mu =\VT _{i_{\mu -1}}\cdots \VT _{i_2}
  \VT _{i_1}^\chi (\Lambda )$
  for all $\mu \in \{1,2,\dots ,\nu \}$.
  By Lemma~\ref{le:VTinv}, the assumptions on $\Lambda $ are
  equivalent to the relations
  \[ \prod _{\mu =1}^{\nu -1} \prod _{m=1}^{\bfun {\chi _\mu }(\al _{i_\mu })-1}
  (\Lambda _\mu (K_{i_\mu }L_{i_\mu }^{-1})
  -\rhomap {\chi _\mu }(\al _{i_\mu })^{m-1}) \not=0 \]
  and $\Lambda _\nu (K_{i_\nu }L_{i_\nu }^{-1})=
  \rhomap {\chi _\nu }(\al _{i_\nu })^{t-1}$. Let
  \[ \beta '_\nu =\s _{i_1}^\chi (\beta _\nu )=1_{\chi _2}\s _{i_2}\s
  _{i_3}\cdots \s _{i_{\nu -1}}(\al _{i_\nu }).\]
  By induction hypothesis there exists a $U(\chi _2)$-submodule $V'$ of
  $M^{\chi _2}(\Lambda _2)$ with
  \begin{align}
    \fch{V'}=\sum _{\al \in \ndN _0^I}\PF ^{\chi _2}(\al ,\beta '_\nu ;t)e^{-\al }.
  \end{align}
  Moreover,
  $\Lambda (K_{i_1}L_{i_1}^{-1})\not=\rho ^\chi (\al _{i_1})^{m-1}$ for all $m\in
  \{1,2,\dots ,\bfun \chi (\al _{i_1})-1\}$, and hence
  $\VTM _{i_1}:M^{\chi _2}(\Lambda _2)\to M^\chi (\Lambda )$ is an isomorphism.
  Let $V=\VTM _{i_1}(V')$.
  By Lemmata~\ref{le:MLmap} and \ref{le:Tpfch}, $V$ is a
  $U(\chi )$-submodule of $M^\chi (\Lambda )$ and
  \begin{align}
    \fch V=\sdot _{i_1}^{\chi _2}(\fch{V'})=\sum _{\al \in \ndN _0^I}
    \sdot _{i_1}^{\chi _2}(\PF ^{\chi _2}(\al ,\beta '_\nu ;t)e^{-\al }).
  \end{align}
  Thus, by Lemma~\ref{le:P1},
  \begin{align*}
    \fch V=&e^{(1-\bfun \chi (\al _{i_1}))\al _{i_1}}\s _{i_1}^{\chi _2}\Big(
    \frac {e^{-t\beta '_\nu}-e^{-\bfun{\chi _2}(\beta '_\nu )\beta '_\nu }}
    {1-e^{-\beta '_\nu }}
    \prod _{\beta \in R_+^{\chi _2}\setminus \{\beta '_\nu \}}
    \frac {1-e^{-\bfun{\chi _2}(\beta )\beta }}
    {1-e^{-\beta }}\Big).
  \end{align*}
  Recall that $\beta '_\nu \not=\al _{i_1}$, since $\nu >1$.
  Moreover,
  \begin{align*}
    &e^{(1-\bfun \chi (\al _{i_1}))\al _{i_1}}\s _{i_1}^{\chi _2}\Big(
    \frac {1-e^{-\bfun{\chi _2}(\al _{i_1})\al _{i_1}}}
    {1-e^{-\al _{i_1}}}\Big)\\
    &\quad =e^{(1-\bfun \chi (\al _{i_1}))\al _{i_1}}
    \frac {1-e^{\bfun\chi (\al _{i_1})\al _{i_1}}}
    {1-e^{\al _{i_1}}}
    =\frac {1-e^{-\bfun\chi (\al _{i_1})\al _{i_1}}}
    {1-e^{-\al _{i_1}}}.
  \end{align*}
  Therefore
  \begin{align*}
    \fch V=&\frac {e^{-t\beta _\nu}-e^{-\bfun{\chi }(\beta _\nu )\beta _\nu }}
    {1-e^{-\beta _\nu }}
    \prod _{\beta \in R_+^\chi \setminus \{\beta _\nu \}}
    \frac {1-e^{-\bfun{\chi }(\beta )\beta }}
    {1-e^{-\beta }}=\sum _{\al \in \ndN _0^I} \PF ^\chi (\al ,\beta _\nu ;t)
    e^{-\al }
  \end{align*}
  by Lemma~\ref{le:P1}. This proves Eq.~\eqref{eq:fchV}.
  
  Since $t>0$, $\fch V\not=\fch M^\chi (\Lambda )$.
  By assumption on $t$, $V_{t\beta _\nu }\not=0$, and hence $V\not=0$.
  Since $V$ is a $\ndZ ^I$-graded $U(\chi )$-submodule of $M^\chi (\Lambda )$,
  the lemma is proven.
\end{proof}

\begin{theor}
  \label{th:Shapdet}
  Let $\chi \in \cX _5$. For all $\al \in \ndN _0^I$,
  the Shapovalov determinant of $U(\chi )$ is the family $(\det ^\chi _\al
  )_{\al \in \ndN _0^I}$, where
  \begin{align}
    \label{eq:det}
    \det \nolimits ^\chi _\al =
    \prod _{\beta \in R^\chi _+} \prod _{t=1}^{\bfun \chi (\beta _\nu )-1}
    (\rhomap \chi (\beta )K_{\beta }
    -\chi (\beta ,\beta )^t L_{\beta })
    ^{\PF ^\chi (\al ,\beta ;t)}.
  \end{align}
\end{theor}

\begin{proof}
  Let $\al \in \ndN _0^I$, $k=\dim U^-(\chi )_{-\al }$,
  and let $\{F'_1,F'_2,\dots ,F'_k\}$ be a basis of $U^-(\chi )_{-\al }$.
  Then $\Shf (F'_i,F'_j)\in \sum _{\beta ,\gamma \in \ndN _0^I,\beta +\gamma
  =\al }\fie K_\beta L_\gamma $ by Lemma~\ref{le:Shfcoeffs}, and hence
  \begin{align}\label{eq:detsummands}
    \det \nolimits ^\chi _\al \in
    \sum _{\beta ,\gamma \in \ndN _0^I,\beta +\gamma =k\al }
    \fie K_\beta L_\gamma .
  \end{align}
  The polynomials
  \[ \rhomap \chi (\beta _\nu )K_{\beta _\nu }
  -\chi (\beta _\nu ,\beta _\nu )^t L_{\beta _\nu }=\LT _{i_1}\cdots
  \LT _{i_{\nu -1}}(\rhomap \chi (\beta _\nu )K_{i_\nu }
  -\chi (\beta _\nu ,\beta _\nu )^t L_{i_\nu }) \]
  are irreducible and pairwise distinct for all $\nu \in \{1,2,\dots ,n\}$
  and $t\in \{1,2,\dots ,\bfun \chi (\beta _\nu )-1\}$. Thus by
  Lemma~\ref{le:P2} it suffices to prove that $\det ^\chi _\al \not=0$
  and that the polynomials
  $(\rhomap \chi (\beta _\nu )K_{\beta _\nu }
  -\chi (\beta _\nu ,\beta _\nu )^t L_{\beta _\nu })
  ^{\PF ^\chi (\al ,\beta _\nu ;t)}$
  are factors of $\det ^\chi _\al $.

  Let $\Fie $ be the algebraic closure of $\fie $ and
  $\To =\maxspec \Uz \ot _\fie \Fie $ the algebraic torus.
  The points of $\To $ are just the $\Fienz $-valued characters of $\Uz $.
  The equation $\det ^\chi _\al =0$ defines a closed affine subvariety
  $\To '_\al \subset \To $.
  Let $\Lambda \in \To $. By definition, $\Lambda \in \To '_\al $
  if and only if
  $\Lambda \Shf : U^-(\chi )_{-\al }\times U^-(\chi )_{-\al }\to \Fie $
  is a degenerate symmetric bilinear form, that is,
  if $I^\chi (\Lambda )_{-\al }\not=0$.
  Thus, by Prop.~\ref{pr:M=L}, $\To '_\al $ is a subset of the
  finite union of irreducible varieties
  \[ \To _{\al ,\nu ,t}=\maxspec (\Uz \ot _\fie \Fie )/
  (\rhomap \chi (\beta _\nu )K_{\beta _\nu }-
  \chi (\beta _\nu ,\beta _\nu )^t L_{\beta _\nu }),\]
  where $\nu \in \{1,2,\dots ,n\}$
  and $t\in \{1,2,\dots ,\bfun \chi (\beta _\nu )-1\}$. By
  Eq.~\eqref{eq:detsummands}
  \[ \det \nolimits ^\chi _\al =
  f\prod _{\nu =1}^n\prod _{t=1}^{\bfun \chi (\beta _\nu )-1}
  (\rhomap \chi (\beta _\nu )K_{\beta _\nu }-
  \chi (\beta _\nu ,\beta _\nu )^t L_{\beta _\nu })^{N_{\nu ,t}}\]
  for some $N_{\nu ,t}\in \ndN _0$ and an element
  $f\in \fie [K_i,L_i\,|\,i\in I]$ which is invertible on $\To $.
  In particular, $\det ^\chi _\al \not=0$. We finish the proof of the theorem
  by showing that $N_{\nu ,t}\ge \PF ^\chi (\al ,\beta _\nu ;t)$ for all
  $\nu \in \{1,2,\dots ,n\}$ and $t\in \{1,2,\dots ,\bfun \chi (\beta _\nu )-1\}$.
  The essential ingredients will be Lemmata~\ref{le:subfch} and
  \ref{le:detXfactor}.

  Let $\nu \in \{1,2,\dots ,n\}$ and
  $t\in \{1,2,\dots ,\bfun \chi (\beta _\nu )-1\}$. Let
  $w=1_\chi \s _{i_1}\s _{i_2}\cdots \s _{i_{\nu -1}}$.
  Then
  \[ \Uz =\fie [K_{w(\al _j)},K_{w(\al _j)}^{-1},
  L_{w(\al _j)},L_{w(\al _j)}^{-1}\,|\,j\in I] \]
  and $w(\al _{i_\nu })=\beta _\nu $. Let
  \[ B =\fie [L_{\beta _\nu },L_{\beta _\nu }^{-1},
  K_{w(\al _j)},K_{w(\al _j)}^{-1},
  L_{w(\al _j)},L_{w(\al _j)}^{-1}\,|\,j\in I \setminus \{i_\nu \}] \]
  and $x=\rhomap \chi (\beta _\nu )K_{\beta _\nu }-
  \chi (\beta _\nu ,\beta _\nu )^t L_{\beta _\nu }$.
  Then
  \[ \Uz \simeq B[x,(x+
  \chi (\beta _\nu ,\beta _\nu )^t L_{\beta _\nu })^{-1}]. \]
  Let
  $X'=(x'_{ij})_{i,j\in \{1,2,\dots ,k\}}\in (\Uz )^{k\times k}$
  with
  \[ x'_{ij}=\Shf (F'_i,F'_j) \quad \text{for all $i,j\in \{1,2,\dots ,k\}$.}
  \]
  Let $l\in \ndZ $ such that $K_{\beta _\nu }^l X'\in B[x]^{k\times k}$, and let
  $X=K_{\beta _\nu }^l X'$.
  By Lemma~\ref{le:subfch} and Eq.~\eqref{eq:LShf}
  there is a non-empty open subset of the variety of $B\simeq
  \Uz /(x)$ such that
  $\mathrm{rk}\,X(0)_p\le k-\PF ^\chi (\al ,\beta _\nu ;t)$
  for all $p$ in this set. By Lemma~\ref{le:detXfactor},
  $\det X=x^{\PF ^\chi (\al ,\beta _\nu ;t)}b'$ for some $b'\in
  B[x]$. In particular, $x^{\PF ^\chi (\al ,\beta _\nu ;t)}$ is a factor of
  $\det ^\chi _\al $, and the proof of the theorem is complete.
\end{proof}

\section{Shapovalov determinants for bicharacters with finite root
systems}
\label{sec:shapdetgen}

In Sect.~\ref{sec:shapdet}
we mainly considered bicharacters $\chi \in \cX _5$.
Here we extend our results to all $\chi \in \cX _3$ with
$\chi (\beta ,\beta )\not=1$ for all $\beta \in R^\chi _+$.

In what follows let $\barcX $ denote
the set of $\Fienz $-valued bicharacters on $\ndZ ^I$. Identify $\barcX $
with $(\Fienz )^{I\times I}$ via
$\chi \mapsto (\chi (\al _i,\al _j) )_{i,j\in I}$ for all $\chi \in \barcX $.
For all $i\in \{1,2,3,4,5\}$
define $\barcX _i\subset \barcX $ in analogy to
Eqs.~\eqref{eq:X1}--\eqref{eq:X5}. Note that $\cX _i=\cX \cap \barcX _i$ for
all $i\in \{1,2,3,4,5\}$.

For all $\beta ,\beta '\in \ndZ ^I$ let $f_{\beta ,\beta '}$ be the rational
function on the affine variety $\barcX =(\Fienz )^{I\times I}$ such that
\[ f_{\beta ,\beta '}(\chi )=\chi (\beta ,\beta ') \quad
\text{for all $\chi \in \barcX $.} \]
Clearly, the functions $f_{\beta ,\beta '}$ with $\beta ,\beta '\in \{\al
_i,-\al _i\,|\,i\in I\}$ generate the algebra $\Fie [\barcX ]$.
Recall that a subset of $\barcX $ is locally closed, if it is
the intersection of an open and a closed subset of $\barcX $.

\begin{propo}\label{pr:Vchi}
  Let $\chi \in \barcX _3$. Assume that $\chi (\beta ,\beta )\not=1$
  for all $\beta \in R^\chi _+$.
  Let $\underline{n}=(n_\beta )_{\beta \in R^\chi _{+\infty }}$ with
  $n_\beta \in \ndN $ for all $\beta \in R^\chi _{+\infty }$.
  Then there exists an ideal $J\subsetneq \Fie [\barcX ]$
  generated by products of polynomials of the form
  \[ q-\prod _{i,j\in I}f_{\al _i,\al _j}^{m_{ij}},\quad
  \text{$q$ is a root of $1$, $m_{ij}\in \ndZ $ for all $i,j\in I$,} \]
  such that the set
  \begin{equation}
    \begin{aligned}
      V^\chi _{\underline{n}}=\{\chi '\in \barcX \,|\,&
      R^{\chi '}_+ =R^\chi _+ ,\,
      \bfun {\chi '}(\beta )=\bfun \chi (\beta ) \text{ for all }
      \beta \in R^\chi _{+\fin },\\
      &\chi '(\beta ,\beta )^n\not=1
      \text{ for all $\beta \in R^\chi _{+\infty }$,
      $1\le n\le n_\beta $}
      \}
    \end{aligned}
    \label{eq:Vchin}
  \end{equation}
  is an open subset of $\maxspec \Fie [\barcX ]/J$.
\end{propo}

\begin{proof}
  We use Lemma~\ref{le:equalrs} and Def.~\ref{de:Cartan}
  to reformulate the equation $R^\chi _+=R^{\chi '}_+$.

  Let $\chi '\in \cG (\chi )$.
  Since $\chi \in \barcX _3$, $\chi '$ is $p$-finite
  for all $p\in I$. Further, $\chi '(\al _p,\al _p)\not=1$ since
  $\chi (\beta ,\beta )\not=1$ for all $\beta \in R^\chi
  _+$, see Eq.~\eqref{eq:w*chi}. Thus
  \[ (\chi '(\al _p,\al _p)^{-c_{pj}^{\chi '}}\chi '(\al _p,\al _j)\chi '(\al
  _j,\al _p)-1)(\chi '(\al _p,\al _p)^{1-c_{pj}^{\chi '}}-1)=0 \]
  for all $p,j\in I$ with $p\not=j$.
  Let $w\in \Hom (\chi ,\chi ')\subset \Hom (\Wg (\chi ))$.
  Identify $w$ with the corresponding element in $\Aut (\ndZ ^I)$ in the usual
  way. Then $\chi '=w^*\chi $ and hence
  \begin{align}\label{eq:chirel}
    (\chi (\gamma _p,\gamma _p)^{-c_{pj}^{\chi '}}\chi (\gamma _p,\gamma _j)
    \chi (\gamma _j,\gamma _p)-1)(\chi (\gamma _p,\gamma _p)
    ^{1-c_{pj}^{\chi '}}-1)=0
  \end{align}
  for all $p,j\in I$ with $p\not=j$, where $\gamma _p=w^{-1}(\al _p)$
  and $\gamma _j=w^{-1}(\al _j)$.
  Let
  \begin{align}
    &\begin{aligned}
      &J'=\big(
      (f_{\gamma _p,\gamma _p}^{-c_{pj}^{w^*\chi }}
      f_{\gamma _p,\gamma _j}
      f_{\gamma _j,\gamma _p}-1)
      (f_{\gamma _p,\gamma _p}^{1-c_{pj}^{w^*\chi }}-1)\,|\\
      &\quad j,p\in I,j\not=p,\,
      w\in \Hom (\chi ,\underline{\,\,}),\,\gamma _p
      =w^{-1}(\al _p),\,\gamma _j=w^{-1}(\al _j)\big)
    \end{aligned}
    \label{eq:J'}
  \end{align}
  and
  \begin{align}
    J=J'+\big(
    f_{\beta ,\beta }^{\bfun \chi (\beta )}-1\,|\,
    \beta \in R^\chi _{+\fin } \big).
    \label{eq:J}
  \end{align}
  Then, by Lemma~\ref{le:equalrs},
  Def.~\ref{de:Cartan}, and Eq.~\eqref{eq:height},
  $V^\chi _{\underline{n}}$ is the set of points $\chi ''\in \maxspec
  \Fie [\barcX ]/J$ such that
  \begin{itemize}
    \item $f_{\beta ,\beta }^n(\chi '')\not=1$
      for all $\beta \in R^\chi _{+\infty }$, $1\le n\le n_\beta $ and
    \item $(f_{\gamma _p,\gamma _p}^m f_{\gamma _p,\gamma _j}
      f_{\gamma _j,\gamma _p}-1)(\chi '')\,
      (f_{\gamma _p,\gamma _p}^{m+1} -1)(\chi
      '')\not=0$
      for all $j,p\in I$, $w\in \Hom (\chi ,\underline{\,\,})$, and
      $m\in \{0,1,\dots ,-c^{w^*\chi }_{pj}-1\}$, where
      $j\not=p$ and $\gamma _p=w^{-1}(\al _p)$, $\gamma _j=w^{-1}(\al _j)$.
  \end{itemize}
  This is clearly an open subset, which proves the proposition.
\end{proof}

\begin{propo}\label{pr:X5dense}
  Let $\chi \in \barcX _3$. Assume that $\chi (\beta ,\beta )\not=1$
  for all $\beta \in R^\chi _+$.
  Let $\underline{n}=(n_\beta )_{\beta \in
  R^\chi _{+\infty }}$, where $n_\beta \in \ndN $ for all $\beta \in
  R^\chi _{+\infty }$.
  Let $V^\chi _{\underline{n}}$ be as in Prop.~\ref{pr:Vchi}.
  Then $\barcX _5\cap V^\chi _{\underline{n}}$ is Zariski dense in
  $V^\chi _{\underline{n}}$.
\end{propo}

\begin{proof}
  Prop.~\ref{pr:Vchi} gives that $V^\chi _{\underline{n}}\subset \barcX _3$
  satisfies the conditions on $V$ in Lemma~\ref{le:gendensity}, where
  $k=|I|^2$ and $\{x_i\,|\,i=1,2,\dots,k\}=\{f_{\al _i,\al _j}\,|\,i,j\in
  I\}$. Since $\barcX _4$ contains all finite sets $V_{n_1,\dots ,n_k}$ in
  Lemma~\ref{le:gendensity}, and $\barcX _5\cap V^\chi _{\underline{n}}=
  \barcX _4\cap V^\chi _{\underline{n}}$ by definition of
  $V^\chi _{\underline{n}}$, the proof is completed.
\end{proof}

Similarly to Eq.~\eqref{eq:PF} define $\PF ^\chi (\al ,\beta _\nu ;t)$ for all
$\chi \in \barcX _3$, $\al \in \ndN _0^I$, $\beta _\nu \in R^\chi _+$, and $t\in
\ndN $ with $t<\bfun \chi (\beta _\nu )$ by
\begin{equation}
\begin{aligned}
  \PF ^\chi (\al ,\beta _\nu ;t)=\Big|\Big\{(m_1,\dots ,m_n)\in \ndN _0^n\,\big|\,
  \sum _{\mu =1}^n m_\mu \beta _\mu =\al ,\,m_\nu \ge t,\quad &\\
  m_\mu <\bfun \chi (\beta _\mu )\quad \text{for all $\mu \in \{1,2,\dots ,n\}$}
  \Big\}\Big|.&
  \label{eq:PF2}
\end{aligned}
\end{equation}

\begin{theor}
  \label{th:Shapdet2}
  Let $\chi \in \cX _3$. Assume that $\chi (\beta ,\beta )\not=1$
  for all $\beta \in R^\chi _+$.
  The Shapovalov determinant of $U(\chi )$ is the family
  $(\det ^\chi _\al )_{\al \in \ndN _0^I}$, where
  \begin{align}
    \label{eq:det2}
    \det \nolimits ^\chi _\al =
    \prod _{\beta \in R^\chi _+}
    \prod _{t=1}^{\bfun \chi (\beta )-1}
    (\rhomap \chi (\beta )K_{\beta }
    -\chi (\beta ,\beta )^t L_{\beta })
    ^{\PF ^\chi (\al ,\beta ;t)}.
  \end{align}
\end{theor}

\begin{proof}
  Let $\al \in \ndN _0^I$. Choose a basis $\{F'_1,\dots ,F'_k\}$ of
  $U^-(\chi )_{-\al }$ consisting of monomials $F_{i_1}F_{i_2}\cdots
  F_{i_l}$, where $k,l\in \ndN _0$ and $i_1,\dots ,i_l\in I$. Identify
  $\oplus _{\beta ,\gamma \in \ndN _0^I,\,\beta +\gamma =\al }\Fie
  K_\beta L_\gamma $ with $\Fie ^N$ for an appropriate $N\in \ndN $.
  By the commutation relations \eqref{eq:KLrel}--\eqref{eq:EFrel} and
  the definition of $\Shf $, the map
  \[ d:\barcX \to \Fie ^N, \quad \chi '\mapsto \det (\Shf (F'_i,F'_j))_{i,j\in
  \{1,2,\dots ,k\}} \]
  is a morphism of affine varieties. Further,
  $d(\chi )\not=0$ by Lemma~\ref{le:Shfcoeffs}, the
  choice of $\{F'_1,\dots ,F'_k\}$, and the nondegeneracy of the
  pairing $\eta $, see Prop.~\ref{pr:sHpdef}(iv).
  Recall the definition of
  $|\beta |$, $\beta \in \ndZ ^I$, from Eq.~\eqref{eq:abs}.
  Restrict $d$ to the set $V^\chi _{\underline{n}}$ defined
  in Prop.~\ref{pr:Vchi}, with $n_\beta =|\al |/|\beta |$ for
  all $\beta \in R^\chi _{+\infty }$. The set
  \[ V'=\{\chi '\in V^\chi _{\underline{n}}\,|\,d(\chi ')\not=0 \} \]
  is open in $V^\chi _{\underline{n}}$ and contains $\chi $.
  Thus by Prop.~\ref{pr:X5dense} the set
  \[ V''=\{\chi '\in \barcX _5\cap V^\chi _{\underline{n}}\,|
  \,d(\chi ')\not=0\} \]
  is Zariski dense in all irreducible components of $V^\chi
  _{\underline{n}}$ containing $\chi $.
  The definition of $V^\chi _{\underline{n}}$
  and the choice of $\underline{n}$ yield that $R^{\chi '}_+=R^\chi
  _+$ and
  $\bfun {\chi '}(\beta )\le \bfun \chi (\beta )$ for all $\chi '\in
  V^\chi _{\underline{n}}$
  and $\beta \in R^\chi _+$. Thus
  $\dim U(\chi ')_{-\al }\le \dim U(\chi )_{-\al }$ for all $\chi '
  \in V^\chi _{\underline{n}}$ by Eqs.~\eqref{eq:PBWbasis},
  \eqref{eq:roots}.
  Hence $d(\chi ')$ is a multiple of
  $\det ^{\chi '}_\al $ for all $\chi '\in V^\chi _{\underline{n}}$.
  By Thm.~\ref{th:Shapdet},
  \begin{align}
    \label{eq:d}
    d(\chi ')=a(\chi ')
    \prod _{\beta \in R^\chi _+} \prod _{t=1}^{\bfun \chi (\beta _\nu )-1}
    (\rhomap \chi (\beta )K_{\beta }
    -\chi (\beta ,\beta )^t L_{\beta })
    ^{\PF ^\chi (\al ,\beta ;t)}
  \end{align}
  for all $\chi '\in V''$, where $a(\cdot )$ is some regular function on
  $\barcX $ which does not vanish on $V''$.
  By the density of $V''$,
  Eq.~\eqref{eq:d} holds for all $\chi '\in V'$ in
  the irreducible components of $V^\chi _{\underline{n}}$
  containing $\chi $,
  and $a(\chi ')\not=0$ for all $\chi '\in V'$ by definition of $V'$.
  In particular, Eq.~\eqref{eq:d} holds for $\chi '=\chi $.
  Thus the theorem is proven.
\end{proof}

\section{Quantized enveloping algebras}
\label{sec:Uqg}

We adapt our main result to quantized enveloping algebras.

Let $I$ be a finite set and let $C=(c_{ij})_{i,j\in I}$
be a symmetrizable Cartan matrix of finite type. Let $\lag $ be
the associated semisimple Lie algebra and $R_+$ the set of positive
roots.
For all $i\in I$ let $d_i\in \ndN $ such that
$d_ic_{ij}=d_jc_{ji}$ for all $i,j\in I$.
Assume that the numbers $d_i$, where $i\in I$, are relatively prime.
Identify $\ndZ ^I$ with the root lattice by considering
$\{\al _i\,|\,i\in I\}$ as the set of simple roots.
Let $(\cdot ,\cdot ):\ndZ ^I\to \ndZ ^I$ be the (positive definite)
symmetric bilinear
form defined by $(\al _i,\al _j)=d_ic_{ij}$.
Let $\rho :\ndZ ^I\to \ndZ $ be the linear form defined by
$\rho (\al _i)=d_i$ for all $i\in I$.

Let $\fie $ be a field, and let
$q\in \fienz $. Assume that
$q^{2m}\not=1$ for all $m\in \ndN $ with
$m\le \max \{d_i\,|\,i\in I\}$.
The quantized enveloping
algebra of $\lag $ is the associative algebra $U_q(\lag )$
generated by the elements $E_i$, $F_i$, $K_i$, and $K_i^{-1}$, where
$i\in I$, and defined by the relations
\begin{gather*}
  K_iK_i^{-1}=K_i^{-1}K_i=1,\quad K_iK_j=K_jK_i,\\
  K_iE_jK_i^{-1}=q^{d_ic_{ij}}E_j,\quad
  K_iF_jK_i^{-1}=q^{-d_ic_{ij}}F_j,\\
  E_iF_j-F_jE_i=\delta_{ij}(K_i-K_i^{-1}),\\
  (\ad E_i)^{1-c_{ij}}(E_j)=0,\quad
  (\ad F_i)^{1-c_{ij}}(F_j)=0 \quad (i\not=j)
\end{gather*}
for all $i,j\in I$.
Here $\ad $ denotes adjoint action:
\begin{align*}
(\ad E_i)(x)=E_ix-K_ixK_i^{-1}E_i,\quad
(\ad F_i)(x)=xF_i-F_iK_i^{-1}xK_i
\end{align*}
for all $i\in I$ and
$x\in \langle E_j,F_j,K_j,K_j^{-1}\,|\,j\in I\rangle $.
Traditionally, in the third line of the defining relations
of $U_q(\lag )$ one inserts a
denominator $q^{d_i}-q^{-d_i}$ on the right hand side, but this
denominator can be eliminated by rescaling \textit{e.\,g.} the
variables $E_i$, $i\in I$.

Assume first that $q$ is not a root of $1$. Then, by
\cite[Ch.\,1]{b-Lusztig93} and \cite[Prop.\,2.10]{inp-AndrSchn02},
$U_q(\lag )\simeq U(\chi )/(K_iL_i-1\,|\,i\in I)$,
where $\chi \in \cX $ with
$\chi (\al _i,\al _j)=q^{d_ic_{ij}}$ for all $i,j\in I$.

Assume now that $q$ is a root of $1$.
Let again $\chi \in \cX $ with
$\chi (\al _i,\al _j)=q^{d_ic_{ij}}$ for all $i,j\in I$.
Then 
\[ U(\chi )/(K_iL_i-1,K_\beta ^{\bfun {}(\beta )}-1\,|\,i\in I,
\beta \in R^\chi _+),\]
where
$\bfun {}(\beta )$ is the order of $q^{(\beta ,\beta )}$
for all $\beta \in R_+$,
is isomorphic to Lusztig's small quantum group $u_q(\lag )$.
This was observed \textit{e.\,g.} in
\cite[Thm.\,4.3]{inp-AndrSchn02} by referring to results of Lusztig,
de Concini, Procesi, Rosso, and M\"uller.

Similarly to Eq.~\eqref{eq:Shf} and Def.~\ref{de:Shapdet}
one defines the Shapovalov form and the Shapovalov determinant
$(\det _\al )_{\al \in \ndN _0^I}$
of $U_q(\lag )$ and $u_q(\lag )$, respectively. Alternatively,
since $K_iL_i$ for $i\in I$ and $K_\beta ^{\bfun{}(\beta )}$
for $\beta \in R^\chi _+$ (the latter only if $q$ is a root of $1$)
are central elements in $U(\chi )$ for all $i\in I$,
the Shapovalov form can also be obtained from the definition in
Sect.~\ref{sec:shapdet} via Lemma~\ref{le:Uzideal}.

\begin{theor}
  \label{th:ShapdetUqg}
  Let $I$, $C$, $(d_i)_{i\in I}$, and $\lag $ as above.
  Let $q\in \fienz $. Assume that
  $q^{2m}\not=1$ for all $m\in \ndN $ with
  $m\le \max \{d_i\,|\,i\in I\}$.

  (i) \cite{inp-dCK90}
  If $q$ is not a root of $1$, then
  the Shapovalov determinant of $U_q(\lag )$ is the family
  $(\det _\al )_{\al \in \ndN _0^I}$, where
  \begin{align}
    \label{eq:detUqg}
    \det \nolimits _\al =
    \prod _{\beta \in R_+}
    \prod _{t=1}^\infty
    (q^{2\rho (\beta )}K_{\beta }
    -q^{t(\beta ,\beta )} K_{\beta }^{-1})
    ^{\PF (\al ,\beta ;t)}.
  \end{align}

  (ii) Assume that $q$ is a root of $1$.
  Then
  the Shapovalov determinant of $u_q(\lag )$ is the family
  $(\det _\al )_{\al \in \ndN _0^I}$, where
  \begin{align}
    \label{eq:detuqg}
    \det \nolimits _\al =
    \prod _{\beta \in R_+}
    \prod _{t=1}^{\bfun{}(\beta )-1}
    (q^{2\rho (\beta )}K_{\beta }
    -q^{t(\beta ,\beta )} K_{\beta }^{-1})
    ^{\PF (\al ,\beta ;t)}.
  \end{align}
\end{theor}

\begin{proof}
  Let $\chi \in \cX $ with
  $\chi (\al _i,\al _j)=q^{d_ic_{ij}}$ for all $i,j\in I$.
  Choose the ideal $J$ in Lemma~\ref{le:Uzideal} as explained above.
  Then one gets the Shapovalov determinants of $U_q(\lag )$
  and $u_q(\lag )$ from the one of $U(\chi )$
  in Thm.~\ref{th:Shapdet2}.
\end{proof}

The second part of Thm.~\ref{th:ShapdetUqg} was proved in
\cite{a-KumLetz97} in the case when
the order of $q$ is prime
and $\fie $ is the cyclotomic field $\mathbb{Q}[q]$.

\section{Appendix}

For the proofs of Thms.~\ref{th:Shapdet} and \ref{th:Shapdet2}
we need some commutative algebra which is considered here.
Let $\Fie $ be an algebraically closed field.

\begin{lemma}
  \label{le:rankX}
  Let $B$ be an integral domain, $x$ an indeterminate,
  $k\in \ndN $, and $X\in B[x]^{k\times k}$.
  Then there exist $s\in \{0,1,\dots ,k\}$, $D_1,D_2\in
  B^{k\times k}$, $D_0\in B[x]^{k\times k}$ and $b\in B\setminus \{0\}$ such that
  $\det D_1,\det D_2\not=0$,
  \begin{align}
    D_1XD_2=xD_0+b\,\mathrm{diag} (\underbrace{1,\dots ,1}_s,0,\dots ,0).
    \label{eq:rankX}
  \end{align}
\end{lemma}

\begin{proof}
  Let $\Frac (B)$ be the field of fractions of $B$. Then there exist
  $s\in \{0,1,\dots ,k\}$ and $D'_1,D'_2\in \Frac (B)^{k\times k}$
  such that $\det D'_1,\det D'_2\not=0$ and
  \[ D'_1X(0)D'_2=\mathrm{diag} (\underbrace{1,\dots ,1}_s,0,\dots ,0).\]
  Let $b_1,b_2\in B\setminus \{0\}$ such that $b_1D'_1,b_2D'_2\in B[x]^{k\times k}$.
  Let $b=b_1b_2$, $D_1=b_1D'_1$, and $D_2=b_2D'_2$. Then
  \[ D_1X(0)D_2=b\,\mathrm{diag} (\underbrace{1,\dots ,1}_s,0,\dots ,0),\]
  and hence the lemma holds for $D_0=D_1X'D_2$, where $X'\in B[x]^{k\times k}$
  such that $X=X(0)+xX'$.
\end{proof}

\begin{lemma}
  \label{le:detXfactor}
  Let $B$ be a finitely generated integral domain over $\Fie $, $x$ an indeterminate,
  $k\in \ndN $, $r\in \{0,1,\dots ,k\}$, and $X\in B[x]^{k\times k}$.
  Assume that $\mathrm{rk}\,X(0)_p\le r$ for all points $p$ in a non-empty Zariski open
  subset of the affine variety of $B$.
  Then $\det X=x^{k-r}b$ for some $b\in B[x]$.
\end{lemma}

\begin{proof}
  By Lemma~\ref{le:rankX} there exist $s\in \{0,1,\dots ,k\}$,
  $b\in B\setminus \{0\}$, $D_1,D_2\in B^{k\times k}$, and
  $D_0\in B[x]^{k\times k}$ such that $\det D_1,\det D_2\not=0$ and
  Eq.~\eqref{eq:rankX} holds. Let $V$ be a non-empty Zariski open subset of the
  affine variety of $B\simeq B[x]/(x)$
  such that $(\det D_1)_p,(\det D_2)_p,
  b_p\not=0$ and $\mathrm{rk}\,X(0)_p\le r$ for all $p\in V$. This exists
  by the assumption on $r$ and since the variety of $B$ is irreducible.
  Then $s\le r$ by Eq.~\eqref{eq:rankX} in the points $(p,0)$ of the variety
  of $B[x]$, where $p\in V$. Therefore
  \[ \det D_1\, \det X \det D_2=x^{k-r}b' \]
  for some $b'\in B[x]$. Since
  $\det D_1,\det D_2\in B$ and $B$ is an integral domain,
  we conclude that $\det X\in x^{k-r}B[x]$.
\end{proof}

For all $k\in \ndN $ let
$\Fie [x_i,x_i^{-1}\,|\,1\le i\le k]$ denote the ring of Laurent
polynomials in $k$ variables.
For all
$M=(m_{ij})_{i,j\in \{1,2,\dots ,k\}}\in \mathrm{GL}(k,\ndZ )$ let
\[ X^{(M)} _i=\prod _{j=1}^kx_j^{m_{ij}},\quad 1\le i\le k. \]
Then the ring endomorphism of
$\Fie [x_i,x_i^{-1}\,|\,1\le i\le k]$ given by $x_i\mapsto X^{(M)}_i$
for all $i\in \{1,2,\dots ,k\}$ is an isomorphism with inverse map
given by $x_i\mapsto X^{(M^{-1})}_i$
for all $i\in \{1,2,\dots ,k\}$.

\begin{lemma}
  Let $k\in \ndN $. Let $J\subsetneq \Fie [x_i,x_i^{-1}\,|\,1\le i\le k]$
  be an ideal
  generated by elements of the form $q-\prod _{i=1}^k x_i^{m_i}$, where
  $m_1,\dots ,m_k\in \ndZ $ and $q\in \Fienz $ is a root of $1$.
  Then $J$ is a finite intersection of ideals of the form
  \begin{align}
    ( X_1^{(M)}-q_1, X_2^{(M)}-q_2,\dots ,X_l^{(M)}-q_l),
    \label{eq:primeid}
  \end{align}
  where $l\in \{0,1,\dots ,k\}$, $q_1,\dots ,q_l\in \Fienz $ are roots of $1$,
  and $M\in \mathrm{GL}(k,\ndZ )$.
  \label{le:torusideal}
\end{lemma}

\begin{proof}
  Proceed by induction on $k$. 
  If $J$ is empty, then the claim is true.
  Assume now that $q-\prod _{i=1}^kx_i^{m_i}$ is one of the generators
  of $J$, where $q$ is
  a root of $1$ and $(m_1,\dots ,m_k)\in \ndZ ^k\setminus \{0\}$.
  Let $m_0=\mathrm{gcd}(m_1,\dots ,m_k)$.
  Let $M'\in \mathrm{GL}(k,\ndZ )$ such that
  $m'_{1i}=m_i/m_0$ for all $i\in \{1,2,\dots ,k\}$.
  Then $(X_1^{(M')})^{m_0}-q\in J$,
  and hence $J$ is the intersection of
  the (finite number of) ideals $J+(X^{(M')}_1-q')$,
  where $q'\in \Fie $, ${q'}^{m_0}=q$. By assumption,
  $J+(X^{(M')}_1-q')$ is generated by $X^{(M')}_1-q'$ and by
  elements of the form $q''-\prod _{i=2}^k (X^{(M')}_i)^{m'_i}$, where
  $m'_2,\dots ,m'_k\in \ndZ $ and $q''\in \Fienz $ is a root of $1$.
  Then the claim follows by the induction hypothesis.
\end{proof}

The ideals in Eq.~\eqref{eq:primeid} are prime ideals of
$\Fie [x_i,x_i^{-1}\,|\,1\le i\le k]$, since the quotient ring
$\Fie [x_i,x_i^{-1}\,|\,1\le i\le k]/J
\simeq \Fie [x_i,x_i^{-1}\,|\,1\le i\le k-l]$ is an integral domain.

\begin{lemma}
  Let $k\in \ndN $. Let $J\subsetneq \Fie [x_i,x_i^{-1}\,|\,1\le i\le k]$
  be an ideal
  generated by polynomials of the form
  \[ \big(q_1-\prod _{i=1}^k x_i^{m_{1i}}\big)
  \big(q_2-\prod _{i=1}^k x_i^{m_{2i}}\big)\cdots
  \big(q_l-\prod _{i=1}^k x_i^{m_{li}}\big), \]
  where $l\in \ndN $, $m_{j1},\dots ,m_{jk}\in \ndZ $ and $q_j\in \Fienz $
  is a root of $1$ for all $j\in \{1,2,\dots ,l\}$.
  Let $V\subset \maxspec \Fie [x_i,x_i^{-1}\,|\,1\le i\le k]/J$ be an open
  subset.
  Then the union of the subsets
  \[ V_{n_1,\dots ,n_k}=\{p\in V\,|\,p_1^{n_1}=1,\dots ,p_k^{n_k}=1\},
  \quad n_1,\dots ,n_k\in \ndN ,\]
  is dense in $V$ with respect to the Zariski topology.
  \label{le:gendensity}
\end{lemma}

\begin{proof}
  We can assume that
  $V=\maxspec \Fie [x_i,x_i^{-1}\,|\,1\le i\le k]/J$.
  Moreover,
  it suffices to prove the lemma for the irreducible components of
  $V$. Thus, as a first reduction, $J$ can be assumed to be as in
  the assumptions of Lemma~\ref{le:torusideal}.
  Then by Lemma~\ref{le:torusideal} we may
  assume that $J$ is an ideal as
  in Eq.~\eqref{eq:primeid}, where $M\in \mathrm{GL}(k,\ndZ )$,
  $l\in \{0,1,\dots ,k\}$, and $q_1,\dots ,q_l\in \Fienz $ are roots
  of $1$.
  Then $V\simeq (\Fienz )^{k-l}$ for some $l\in \{0,1,\dots ,k\}$,
  and the lemma follows from the fact that infinite subsets are dense
  in $\Fienz $.
\end{proof}


\providecommand{\bysame}{\leavevmode\hbox to3em{\hrulefill}\thinspace}
\providecommand{\MR}{\relax\ifhmode\unskip\space\fi MR }
\providecommand{\MRhref}[2]{%
  \href{http://www.ams.org/mathscinet-getitem?mr=#1}{#2}
}
\providecommand{\href}[2]{#2}

\end{document}